\documentclass[final,onefignum,onetabnum]{siamart220329}

\usepackage{lipsum}
\usepackage{amsfonts}
\usepackage{amsmath}
\usepackage{amssymb}
\usepackage{graphicx}
\usepackage{epstopdf}
\usepackage{algorithmic}
\usepackage{graphicx}
\usepackage{subcaption}
\usepackage[linesnumbered,algoruled,boxed,lined,algo2e]{algorithm2e}

\ifpdf
  \DeclareGraphicsExtensions{.eps,.pdf,.png,.jpg}
\else
  \DeclareGraphicsExtensions{.eps}
\fi

\newsiamremark{remark}{Remark}
\newsiamremark{hypothesis}{Hypothesis}
\crefname{hypothesis}{Hypothesis}{Hypotheses}
\newsiamthm{claim}{Claim}

\headers{Dynamics of measure-valued agents in the space of probabilities}{G. Borghi, M. Herty, and A. Stavitskiy}

\title{Dynamics of measure-valued agents\\ in the space of probabilities\thanks{Submitted to the editors \today.
\funding{The work of G.B. and A.S. is funded by the Deutsche Forschungsgemeinschaft (DFG, German Research Foundation) through 320021702/GRK2326 ``Energy, Entropy, and Dissipative Dynamics (EDDy)''. 
M.H.  thanks the Deutsche Forschungsgemeinschaft (DFG, German Research Foundation) for the financial support through 320021702/ GRK2326,  333849990/IRTG-2379, CRC1481, HE5386/18-1,19-2,22-1,23-1, ERS SFDdM035 and under Germany’s Excellence Strategy EXC-2023 Internet of Production 390621612 and under the Excellence Strategy of the Federal Government and the Länder.}}}

\author{Giacomo Borghi\thanks{Heriot-Watt University, School of Mathematical
and Computer Sciences, Currie EH14 4AP, Edinburgh, United Kingdom (\email{g.borghi@hw.ac.uk}, corresponding author)}
\and Michael Herty\thanks{RWTH Aachen University, Institut für Geometrie und Praktische Mathematik, Templergraben 55, 52062, Aachen, Germany (\email{herty@igpm.rwth-aachen.de})}
\and Andrey Stavitskiy\thanks{RWTH Aachen University, Institut für Mathematik, Templergraben 55, 52062, Aachen, Germany (\email{stavitskiy@eddy.rwth-aachen.de})}}

\usepackage{amsopn}

\ifpdf
\hypersetup{
  pdftitle={Manuscript},
  pdfauthor={G. Borghi, M. Herty, and A. Stavitskiy}
}
\fi

\def\RR{{\mathbb R}}
\def\T{\mathrm {T}}
\def\N{\mathbb N}

\def\Rd{{\mathbb {R}^\textup{d} }}

\def\Leb{{\mathcal {L} }}
\def\Lebo{{\mathcal {L}^1 }}
\def\diam{{\mathrm {diam} }}
\def\supp{\mathrm{supp}}

\def\ve{\varepsilon}

\def\E{\mathcal{E}}
\def \d {\textup{d}}
\def\Lip{\mathrm{Lip}}
\def\Bar{\mathrm{Bar}}
\def\Sol{\mathrm{Sol}}
\def\Cl{\mathrm{Cl}}
\def\Id{\textup{Id}}

\begin{document}

\maketitle

\begin{abstract}
Motivated by the development of dynamics in probability spaces, we propose a novel multi-agent dynamic of consensus type where each agent is a probability measure.  The agents move instantaneously towards a weighted barycenter of the ensemble according to the 2-Wasserstein metric. We mathematically describe the evolution as a system of measure differential inclusions and show the existence of solutions for compactly supported initial data. Inspired by the consensus-based optimization, we apply the multi-agent system to solve a minimization problem over the space of probability measures. In the small numerical example, each agent is described by a particle approximation and aims to approximate a target measure.
\end{abstract}

\begin{keywords}
multi-agent systems, Wasserstein space, consensus dynamics, measure differential inclusions, consensus-based optimization
\end{keywords}

\begin{MSCcodes}
34G20, 35D30, 
49Q22
\end{MSCcodes}

\section{Introduction}

Multi-agent dynamics are mathematical models used to describe the interaction between several undistinguishable entities, called \textit{agents}, over time. The many fields of application of dynamical multi-agent systems range from sociology \cite{friedkin1990}, finance \cite{Samanidou2007}, biology \cite{carrillo2019biology}, to sensor networks \cite{yu2009sensor} and optimization \cite{ma2019multi}. In this work, we are particularly interested in multi-agent dynamical systems as an algorithmic process where a complex task is solved by the agents' ensemble via simple interaction rules, possibly involving random choices. Consensus dynamics \cite{becchetti2020consensus} refers to systems where the agents, starting from different initial configurations, aim to self-organize and reach an agreement, or \textit{consensus}, among a common value, that could be, for instance, a simple average of the initial configurations or even a solution to an optimization problem \cite{pinnau2017consensus}. In this context, having a mathematical formulation of the multi-agent dynamics is essential to provide analytical insights into the self-organizing properties of the system. One is interested, typically, in investigating stationary states of the dynamics and their stability against small perturbations. Another property to analyze, of course, is the capability of the agents to converge to a consensus configuration and, eventually, the time needed to achieve it. We refer to the recent survey \cite{becchetti2020consensus} for more details on consensus dynamics.

Depending on the application, the agent configurations usually consist of a discrete set of labels or a continuous space $\Rd$ for some space dimension $\d \in \mathbb{N}$. Here, however, we consider agents described by Borel probability measures with bounded second moments $\mathcal{P}_2(\Rd)$, as they allow to model many interesting objects such as large data sets, images, and noisy measurements. The aim is to derive a multi-agent system that is directly defined at such an infinite dimensional level and that does not rely on a specific numerical representation of the agents. In particular, we equip $\mathcal{P}_2(\Rd)$ with the 2-Wasserstein distance $W_2(\cdot, \cdot)$ and exploit the structure given by such metric. Wasserstein distances are related to the problem of optimal transportation between probability measures \cite{santambrogio2015optimal} and have recently gained great popularity in different applied fields such as unsupervised learning \cite{wassGAN}, statistics \cite{Panaretos2020}, and image processing \cite{Bonneel2023imaging}. 

Consider $N$ agents with configurations $\mu^i_t \in \mathcal{P}_2(\Rd)$, $i =1, \dots, N$ at time $t\geq 0$. In Euclidean settings, the agents' consensus interaction tend to average their configurations to reach full synchronization. If the communication network is fully connected, this consists of a movement towards a common, possibly weighted, average of their configurations. The concept of average is generalized in metric spaces by the concept of  Fr{\'e}chet mean \cite{frechet1948elements}, or barycenter, which in our settings consists of a solution to 
\begin{equation}
\underset{\nu \in \mathcal{P}(\Rd)}{\min} \sum_{i=1}^N \lambda(\mu^i_t)\, W_2^2\left(\mu^i_t , \nu\right)
\label{eq:Nbary}
\end{equation}
known as weighted Wasserstein barycenter \cite{carlier2011bary}, where $\lambda: \mathcal{P}_2(\Rd) \to [0,\infty)$ is a weight function that describes the influence of each agent $\mu_t^i$ in the dynamics. We aim to model a dynamics where every agent instantaneously moves towards the weighted barycenter along one of its geodesics. 

To describe such evolution, we rely on the framework of Measure Differential Equations (MDE) and Inclusions (MDI) introduced in \cite{piccoli2019measure,piccoli2018mdi}.
Fix the weight function $\lambda$, and consider $N$ agents described by absolutely continuous measures $\mu_t^i$. We denote with $T_{\mu^i_t}^{\overline{\mu}_t}$ the optimal transport map between the agent $\mu^i_t$ and the barycenter $\overline{\mu}_t$ weighted by the coefficients $\lambda(\mu_t^i)$.
In this paper, we show in Theorem \ref{t:ac} the existence of solutions to the consensus multi-agent system satisfying the MDEs 
\begin{equation}
\begin{cases}
\overline \mu_t = \underset{\nu \in \mathcal{P}(\Omega)}{\textup{argmin}} \sum_{j=1}^N \lambda(\mu_t^j)\, W_2^2(\mu^j_t , \nu) \\
\frac{d}{dt}{\mu}_t^i  +\mathrm{div}\left[ \left( T_{\mu^i_t}^{\overline{\mu}_t}(x) - x \right) \mu_t^i(dx)\right]=0 \quad \forall\; i = 1,\dots,N
\end{cases}
\label{eq:introMDE}
\end{equation}
in the sense of distributions, for compactly supported and absolutely continuous initial agents $\mu_0^i\in \mathcal{P}_2(\Rd)$, $i = 1, \dots,N$.
More specifically, we construct solutions starting from an explicit, time-discrete iteration.
This is particularly relevant as we are interested in the mathematical modelling of numerical algorithms based on computational agents.

Then, we extend the model to arbitrary measures which may not be absolutely continuous and substitute transport maps with measures that are solutions to the optimal transport problem that may, in general, not be unique.
Let $\Gamma_o(\mu_t^i, \overline{\mu}_t)$ be the set of optimal couplings for the quadratic cost.
Consider the projections $\pi_0,\pi_1$ defined as $\pi_0 (x,y) = x$ and $\pi_1(x,y) = y$ for any $x,y\in \Rd$. For compactly supported initial agents $\mu^i_0 \in\mathcal{P}_2(\Rd)$, $i =1,\dots,N$ we also show the existence of solutions to the system of MDIs \begin{equation}
\begin{cases}
\overline \mu_t \in \underset{\nu \in \mathcal{P}(\Omega)}{\textup{argmin}} \sum_{j=1}^N \lambda(\mu_t^j)\, W_2^2(\mu^j_t , \nu) \\
\frac{d}{dt}{\mu}_t^i  +\mathrm{div}\left[ \left( y - x \right ) \gamma_t^i(dx,dy)\right]=0 \quad \textup{for some}\;\; \gamma_t^i \in \Gamma_o(\mu_t^i, \overline \mu_t), \quad \forall\; i = 1,\dots,N\,,
\end{cases}
\label{eq:introMDI}
\end{equation}
in the sense of Definition \ref{def:mdi}, see Theorem \ref{t:mdi} which is the main result of the paper.

In the time-discrete settings, similar dynamics have been proposed in \cite{doucet2014,doucet2021,gualandi2020, bullo2023} as distributed algorithms for the computation of Wasserstein barycenters. In \cite{doucet2014,doucet2021}, authors consider probability measure on $\RR$, and a directed communication graph. At every step, the algorithm substitutes each agent with a barycenter computed among its neighbours, but there is no concept of moving along geodesics. The dynamics suggested in \cite{gualandi2020,bullo2023}, instead, describe pairwise interaction between agents in $\mathcal{P}(\Rd)$, where the pairs of agents move along the geodesics connecting them. In \cite{bullo2023} those couples are picked by randomly selecting edges of a given weighted communication graph, while in \cite{gualandi2020} the interaction between the agents follows a given order.  To our knowledge, this paper represents the first attempt to provide a time-continuous mathematical model for multi-agent consensus dynamics in the 2-Wasserstein space.

As \eqref{eq:introMDE}, \eqref{eq:introMDI} do not rely upon a specific representation of $\mathcal{P}_2(\Rd)$, the dynamics can be simulated by using agents described by point clouds, histograms, kernel densities, or, in more complex settings, generative models \cite{goodfellow2014}. We note that modern optimal transport algorithms allow fast computation of barycenters and geodesics in such settings and even for relatively high-dimensional spaces \cite{cuturi2019computational,gero2023gencol}.

The paper is organized as follows. In Section \ref{sec:definitions} we recall basic notions related to Wasserstein distances and barycenter and recall the framework of MDE and MDI. We also show how reparametrized geodesics can be seen as solutions to MDIs. Next, we prove the existence of solutions to \eqref{eq:introMDE} (absolutely continuous case), in Section \ref{Sec:DiffInclusion}, and solutions to \eqref{eq:introMDI} (general case) in Section \ref{Sec:mainresult}. In Section \ref{sec:num} we present an application of the multi-agent system inspired by Consensus-Based Optimization methods \cite{pinnau2017consensus}. Final remarks, also on possible future research perspectives, are discussed in the last section.

\section{Preliminaries}
\label{sec:definitions}

\subsection{Wasserstein spaces}

Given a metric space $(M,d)$ we will denote with $\mathcal{M}_b^+(M)$ the space of positive bounded Borel measures on it, and we denote with $\mathcal{P}(M)$ the set of Borel probability measures supported over $M$, and with $\mathcal{P}_2(M)$ those with finite second moment.
For any $\d>0$ we indicate with $\mathcal{L}^\d$ the Lebesgue measure over the space $\Rd$ and, given a subset $E\subseteq\Rd$, we denote with $\mathcal{P}^{a.c.}(E)$ and $\mathcal{P}_2^{a.c.}(E)$ the subsets of $\mathcal{P} (E)$ and $\mathcal{P}_2 (E)$ respectively only containing measures which are absolutely continuous with respect to Lebesgue measure $\mathcal{L}^\d$.
Given $\textup n,\textup m>0$ and a couple of subsets $E_1\subseteq\RR^{\textup n}$ and $E_2\subseteq\RR^{\textup m}$, for any measurable function $\phi: E_1 \to E_2$, we define the push-forward
\begin{align*}
\phi_\#:\mathcal{P} (E_1)&\to \mathcal{P}(E_2)\\
\mu&\mapsto \mu\circ\phi^{-1}.
\end{align*}
We will also adopt the distributional notation and therefore for any $\varphi\in C^\infty_c(\Rd)$ we will write
\[
\langle\varphi,\mu\rangle = \int_{\Rd}\varphi(x)\,\mu(dx).
\]
Given a subset $\Omega\subseteq\Rd$ we will  view $\mathcal{M}^+_b(\Omega)$ as a subset of $\mathcal{M}^+_b(\Rd)$, and all its subsets mentioned above, by extending all measures in $\mathcal{M}^+_b(\Omega)$ by zero. Therefore, all $\mu\in \mathcal{M}^+_b(\Omega)$ will be $\mu\in \mathcal{M}^+_b(\Rd)$ with $\supp(\mu)\subseteq\Omega$.

We equip $\mathcal{P}_2(\Rd)$ with the $2$-Wasserstein distance:
\[ 
W_2(\mu, \nu) := \left(\underset{\gamma \in \Gamma(\mu, \nu)}{\min} \int_{\Omega \times \Omega} |x - y|^2 \gamma(dx,dy)\right)^{1/2} \quad \mu, \nu \in \mathcal{P}_2(\Omega)\,,
\]
where with $\Gamma(\mu, \nu)$ we denote the set of Borel probability measures having $\mu$ and $\nu$ as first and second marginal, respectively. The set of optimal couplings realizing the $2$-Wasserstein distance is denoted with $\Gamma_o(\mu, \nu)$. We recall that, in general, there is no unique optimal coupling between two measures, unless one of them admits a Lebesgue density \cite{santambrogio2015optimal}. In this case, the unique optimal coupling is given by $(\textup{Id},T_{\mu}^\nu)_{\#} \mu$, with $T_\mu^\nu:\Rd \to \Rd$ being the optimal transport map between $\mu$ and $\nu$, that is, $\left(T_{\mu}^\nu\right)_\#\mu=\nu $ and
\[
W_2(\mu, \nu)= \left(\int_\Omega |x-T_{\mu}^\nu(x)|^2\mu(dx)\right)^{1/2}.
\]
We will use the fact that the barycenter problem \eqref{eq:Nbary} is related to a multi-marginal optimal transport problem and so we introduce a more general notation for $\Gamma$.
For any $N\in\N$ and $\mu_i\in\mathcal{P}_2(\Rd)$ for all $i=1,\dots,N$, we define
\[
\Gamma(\mu_1,\dots,\mu_N)=\left\{\gamma\in\mathcal{P}\left((\Rd)^N\right)\,\middle|\, (p_i)_\#\gamma=\mu_i,\, i=1,\dots,N\right\}\,,
\]
where $p_i: (\Rd)^N \to \Rd$ denotes the projection to the $i$-th coordinate. 
We recall the definition of metric barycenters first introduced by Fr\'echet in the context of random variables on metric spaces \cite{frechet1948elements} and more recently studied in the context of Wasserstein spaces in \cite{carlier2011bary}. We will localize to 2-Wasserstein spaces thus defining the barycenter of an arbitrary probability measure $\mathbb{P} \in \mathcal{P}(\mathcal{P}_2(\Rd))$.
\begin{definition} 
Given a probability measure $\mathbb{P} \in \mathcal{P}(\mathcal{P}_2(\Rd))$ the 2-Wasserstein barycenters are the solutions to the minimization problem
\begin{equation}
\mu_{\mathbb P} \in \underset{\nu \in \mathcal{P}_2(\Omega)}{\textup{argmin}}  \int_{\mathcal{P}_2(\Omega)} W_2^2(\mu, \nu) \mathbb P(d\mu)\,.
\end{equation}
\end{definition}
We note that the set of solutions to the minimization in the barycenter problem is invariant under multiplication of the functional by constants, therefore in this paper we will adopt the following slightly more general definition.

\begin{definition} \label{def:bary}
For $\mathbb{Q} \in\mathcal{M}_b^+(\mathcal{P}_2(\Rd))$ a positive, bounded measure over $\mathcal{P}_2(\Rd)$, the set of 2-Wassersetein barycenters is given by
\begin{equation}
\Bar(\mathbb Q) := \left \{\mu_{\mathbb Q}\in \mathcal{P}_2(\Rd)\;\middle|\; \mu_{\mathbb Q} \in \underset{\nu \in \mathcal{P}_2(\Rd)}{\textup{argmin}}  \int_{\mathcal{P}_2(\Rd)} W_2^2(\mu, \nu)\, \mathbb Q(d\mu)\right \}.
\end{equation}
\end{definition}
Thanks to the above definition, for a given $\mathbb P \in \mathcal{P}(\mathcal{P}_2(\Rd))$, and for any non-zero $\lambda\in C_b( \mathcal{P}(\Rd) , (0, \infty))$, the set of weighted barycenters is denoted by
\[
\Bar(\lambda \mathbb P) = \left \{\mu_{\lambda\mathbb P}\in \mathcal{P}_2(\Rd)\;\middle|\; \mu_{\lambda\mathbb P} \in \underset{\nu \in \mathcal{P}_2(\Rd)}{\textup{argmin}}  \int_{\mathcal{P}_2(\Rd)} W_2^2(\mu, \nu)\lambda(\mu)\, \mathbb P(d\mu)\right \}\,.
\]
Uniqueness of the barycenter is guaranteed in presence of absolutely continuous measure, that is, if $\mathbb{P}(\mathcal{P}_2^{a.c.})>0$.
For more details on Wasserstein barycenters and their properties, we refer to \cite{carlier2011bary,carlier2022quantitative,le2017existence}.

Finally, we will be working with sets of $N\in \mathbb{N}$ agents $\mu^i \in \mathcal{P}_2(\Rd)$ for $i =1, \dots, N$, which we will also identify to a measure $\mu\in \mathcal{P}_2(\RR^{\d N})$ as $\mu=\bigotimes_{i=1}^N\mu^i$. For any $\mu\in \mathcal{P}_2(\RR^{\d N})$ we introduce the associated empirical measure $\mathbb{P}_{\mu} \in \mathcal{P}(\mathcal{P}_2(\Rd))$
\begin{equation}
\mathbb{P}_{\mu} := \frac 1N \sum_{i=1}^N \delta_{\mu^i}
\label{eq:empirical}
\end{equation}
where $\mu^i = (p_i)_\#\mu$ is the $i$-th marginal of $\mu$. Thanks to the notation we introduced, the set of weighted barycenters among $N$ agents \eqref{eq:Nbary} can be simply denoted by $\Bar(\lambda \mathbb{P}_{\mu})$.


\subsection{Measure differential equations and inclusions}

We begin this section by recalling definitions presented in \cite{piccoli2019measure}.
Let $\T\Rd=\Rd\times\Rd$ denote the tangent bundle of $\Rd$ and $p_1:\T\Rd \to \Rd$ the projection to its base, $p_1(x,v) = x$.
On $\T\Rd$ we will consider the Euclidean metric.

A vector field on $\Rd$ is a map $X:\Rd\to\T\Rd$ such that $p_1\circ X(x)=x$ for all $x\in\Rd$. A Probability Vector Field (PVF) is a map
\[
V:\mathcal{P}_2(\Rd) \to \mathcal{P}_2(\T\Rd) 
\] 
such that $(p_1)_\#V[\mu] = \mu$ for all $\mu\in\mathcal{P}_2(\Rd)$.
We view measures $W\in\mathcal{P}_2(\T\Rd)$ as distributions of order one through the operator $T_W$
\begin{align*}
T_W:C^\infty_c(\Rd)&\to \RR\\
\varphi&\mapsto \int_{\T\Rd}v\cdot\nabla\varphi(x) W(dx,dv)\addtocounter{equation}{1}\tag{\theequation}\label{eq:distribution}
\end{align*}
which is continuous on $C^1_c(\Rd)$ since 
\[
|T_W(\varphi)|\leq \|\nabla\varphi\|_\infty\left(\int_{\T\Rd}|v|^2W(dx,dv)\right)^{1/2}.
\]
We will use this to define the notation 
\[
\langle\nabla\varphi,W\rangle:= T_W(\varphi).
\]
Due to the above definition of a PVF we consider integral curves  by the following definition.
\begin{definition}[Measure Differential Equation]
Let $V:\mathcal{P}_2(\Rd)\to\mathcal{P}_2(\T\Rd)$ be a PVF and $T>0$ a time horizon.
Given a measure $\tilde \mu \in\mathcal{P}_2(\Rd)$, the Cauchy problem defined by $V$ and $\tilde\mu $ is the equation
\begin{equation}
\left\{\begin{array}{ll}
\dot{\mu_t}  = V[\mu_t]\qquad\Lebo\mbox{-a.e. }t\in[0,T]\\
\mu_0=\tilde\mu\,,
\end{array}\right.
\label{eq:MDE}
\end{equation}
understood in distributional sense.

Consequently, we say that a curve $\mu\in C([0,T],\mathcal{P}_2(\Rd))$ is an integral curve for the PVF $V$ if it satisfies, for any $\varphi \in C_c^\infty(\Rd)$, the following
\begin{equation}
\frac{d}{dt}\int_\Rd \varphi(x) \mu_t(dx)= \int_{\T\Rd}v\cdot\nabla \varphi(x)\,V[\mu_t](dx,dv) \quad  \Lebo \mbox{-a.e. }t\in[0,T]
\label{eq:MDEweak}
\end{equation}
with $\lim_{t\to0}W_2(\mu_t,\tilde\mu)=0$.
We will refer to \eqref{eq:MDEweak} as a Measure Differential Equation (MDE).
\end{definition}
We note that, thanks to the notation introduced in the preceding definitions, we can shorten the equation in \eqref{eq:MDEweak} to
\[
\frac{d}{dt}\langle \varphi , \mu_t\rangle =\langle \nabla \varphi , V[\mu_t]\rangle 
\,.
\]

\begin{remark} Observe that  the integrand depends linearly on the variable $v\in \T_x\Rd$. Further, we have $(p_1)_\#V[\mu] = \mu$, and can disintegrate $V[\mu]$ along $p_1$ obtaining the existence of a measurable function $\nu^\mu:\Rd\to\mathcal{P}(\Rd)$ such that $V[\mu]=\nu^\mu_x\otimes\mu$.
We refer to \cite{ambrosio2000functions} for more details on the disintegration of measures. This allows us to define, for any $\mu\in\mathcal{P}_2(\Rd)$, a measurable function $X_{\mu}:\Rd\to\Rd$ through
\[
X_{\mu}(x)=\int_\Rd v\,\nu^\mu_x(dv)\, .
\]
This point of view allows us to write \eqref{eq:MDEweak} as
\[
\frac{d}{dt}\int_\Rd \varphi(x) \mu_t(dx)=\int_{ \Rd}X_{\mu_t}(x)\cdot\nabla \varphi(x)\, \mu_t (dx)\,
\]
which will be useful in Lemma \ref{lemma:system2one}.
We also point out how the distribution $T_V$ appearing in \eqref{eq:distribution} can be represented explicitly as the divergence of a current $\mathrm{div}(X (p_1)_\#V)$. 
\end{remark}

We aim to describe the multi-agent dynamics via MDEs  \cite{piccoli2019measure}. The instantaneous movement of each agent is determined by the geodesic connecting the agent to the barycenter of the ensemble.
Geodesics in the Wasserstein space are parametrized by the optimal couplings between the measures representing the agents and their barycenters. However, those couplings are, in general, not unique. To take this into account, we define a Multi-valued PVF and consider a Measure Differential Inclusion (MDI) \cite{piccoli2018mdi}.

A Multi-valued Probability Vector Field (MPVF) is a set-valued map
\begin{align*}
\mathcal{V}: \mathcal{P}_2(\Rd) 	\rightrightarrows \mathcal{P}_2(T\Rd)
\end{align*}
satisfying $(p_1)_\#V  = \mu$ for every $V \in \mathcal{V}[\mu]$ and every $\mu\in\mathcal{P}_2(\Rd)$.
For notational simplicity, we define for all $\varphi\in C^\infty_c(\Rd)$ and all $\mu\in\mathcal{P}_2(\Rd)$ 
\begin{align*}
\left\langle\nabla\varphi,\mathcal{V}[\mu]\right\rangle := \left \{ \int_{\T\Rd}\nabla\varphi(x)\cdot v\,V(dx,dv)\;:\; V \in \mathcal{V}[\mu] \right\}
\end{align*}
and consequently for any curve $\mu\in C\left([0,T],\mathcal{P}_2\left(\Rd\right)\right)$
\begin{align*}
\int_0^T\left\langle\nabla\varphi,\mathcal{V}[\mu_t]\right\rangle dt:= \left \{\int_0^T \int_{\T\Rd}\nabla\varphi(x)\cdot v\,V_t(dx,dv)dt\;:\; V_t \in \mathcal{V}[\mu_t]\;\mbox{measurable} \right\}.
\end{align*}

This framework has been studied in \cite{piccoli2018mdi,bonnet2021diff,savare2023diss}. However, the multi-agent system we aim to describe will provide us only with upper semi-continuous MPVF. Therefore, we cannot apply the well-posedness results provided in \cite{piccoli2018mdi}, which considers continuous MPVF.

We recall the definition of upper semi-continuity for multivalued functions and MDIs.
\begin{definition}[Upper semi-continuity] \label{def:uppersemi}
We say that a MPVF $\mathcal{V}: \mathcal{P}_2(\Rd) \rightrightarrows \mathcal{P}_2(\T\Rd)$ is upper semi-continuous if its graph is closed, i.e., for any couple of sequences of measures $(\mu_n,\rho_n)\in\mathcal{P}_2(\Rd)\times \mathcal{P}_2(\T\Rd)$ such that $\mu_n\to \mu\in\mathcal{P}_2(\Rd)$ and $\rho_n\to \rho\in\mathcal{P}_2(\T\Rd)$ (in 2-Wasserstein metric), we have that if $\rho_n\in\mathcal{V}[\mu_n]$ then we must have $\rho\in\mathcal{V}[\mu]$.
\end{definition}

\begin{definition}[Measure Differential Inclusion] \label{def:mdi}
Let $\mathcal V: \mathcal{P}_2(\Rd)\rightrightarrows \mathcal{P}_2(T\Rd)$ be a MPVF and $T>0$ a time horizon.
Given a measure $\tilde\mu \in\mathcal{P}_2(\Rd)$, we define the Cauchy problem as the following distributional inclusion
\begin{equation}
\left\{\begin{array}{ll}
\dot{\mu}_t  \in \mathcal{V}[\mu_t] \quad \Lebo\mbox{-a.e. } t\in[0,T]\\
\mu_0=\tilde{\mu}\,.
\end{array}\right.
\label{eq:MDI_simple}
\end{equation}
A curve $\mu\in C([0,T],\mathcal{P}_2(\Rd))$ is a solution for the Cauchy problem defined by $\mathcal {V}$ and $\tilde \mu$ if it satisfies for every $\varphi \in C_c^\infty(\Rd)$ and almost every $t\in[0,T]$
\begin{equation}
\frac{d}{dt}\int_\Rd \varphi(x)\mu_t(dx)\; \in \;\left \{ \int_{T\Rd}(v\cdot\nabla \varphi(x))\,V(dx,dv)\;:\; V \in \mathcal{V}[\mu_t] \right\}\,.
\label{eq:MDIweak}
\end{equation}
\end{definition}

We will also write the inclusion \eqref{eq:MDIweak} for almost every $t\in[0,T]$ as
\[
\frac{d}{dt}\left\langle\varphi,\mu_t\right\rangle\in\left\langle\nabla\varphi,\mathcal{V}[\mu_t]\right\rangle 
\]
or, for any $t\in[0,T]$,
\[
\left\langle\varphi,\mu_t- \tilde \mu \right\rangle \in\int_0^t\left\langle\nabla\varphi,\mathcal{V}[\mu_t]\right\rangle\,dt\, .
\]

The Wasserstein metric behaves like a strong norm with respect to the weak topology, thanks to  \cite{hille2021equivalence}. For any two measures $\mu,\nu\in \mathcal{P}_2(\Rd)$ and $\varphi\in C^\infty_c(\Rd)$ we have therefore the inequality
\[
\left|\langle\varphi,\mu-\nu\rangle\right|\leq \|\nabla\varphi\|_\infty W_2(\mu,\nu)\,.
\]
In a similar fashion, we have the following Lemma.

\begin{lemma}\label{Lem:VectorFieldBound} 
Given two measures $V, W\in\mathcal{P}_2(\T\Rd)$, then it holds for some $C>0$
\[
\left|\langle\nabla\varphi,V -W\rangle\right|\leq C\|\nabla\varphi\|_{W^{1,\infty}} W_2(V,W)\,\quad \textup{for all} \;\;\varphi\in C^\infty_c(\Rd).
\]

\end{lemma}
\begin{proof}
Let us consider $\Sigma\in\Gamma_o(V,W)$ and compute
\begin{align*}
\left|\langle\nabla\varphi,V-W\rangle\right|&=\left|\int_{T\Rd}\nabla\varphi(x)\cdot v \,(V-W)(dx,dv)\right|\\
&=\left|\int_{\T\Rd\times\T\Rd}\nabla\varphi(x)\cdot v-\nabla\varphi(y)\cdot w \,\Sigma(dx,dv,dy,dw)\right|\\
&=\left|\int_{\T\Rd\times\T\Rd}(\nabla\varphi(x)-\nabla\varphi(y))\cdot v+\nabla\varphi(y)\cdot(v- w) \,\Sigma(dx,dv,dy,dw)\right|\\
&\leq \int_{\T\Rd\times\T\Rd}\left|\nabla\varphi(x)-\nabla\varphi(y)\right|\left|v\right|+\left|\nabla\varphi(y)\right|\left|v- w\right| \,\Sigma(dx,dv,dy,dw) \\
&\leq \int_{\T\Rd\times\T\Rd}\|\nabla^2\varphi\|_\infty\left| x - y \right|\left|v\right|+\left\|\nabla\varphi\right\|_\infty\left|v- w\right| \,\Sigma(dx,dv,dy,dw) \\
&\leq \|\nabla^2\varphi\|_\infty \left(\int |x-y|^2\Sigma (dx,dv,dy,dw) \right)^{1/2}\left(\int_{\T\Rd}|v|^2V(dx,dv) \right)^{1/2} \\
& \qquad +\|\nabla\varphi\|_\infty \left(\int |v-w|^2\Sigma(dx,dv,dy,dw) \right)^{1/2}.
\end{align*}
Let
\[
C=\max\left\{1,\left(\int_{\T\Rd}|v|^2V (dx,dv)\right)^{1/2}\right\},
\]
we conclude that
\[
\left|\langle\nabla\varphi,V-W\rangle\right|\leq C\|\nabla\varphi\|_{W^{1,\infty}(\Rd)}W_2(V,W).
\]
\end{proof}

\subsection{Geodesics are solutions to MDIs}\label{SubSec:GeodesicSolutions}
Consider two measures $\mu_0$, $\overline{\mu}\in \mathcal{P}_2(\Rd)$ and the convex combinations $\pi_\tau(x,y) :=(1 - \tau)x + \tau y$ for all $\tau \in [0,1]$ and $x,y\in\Rd$.
We fix $\gamma\in \Gamma_o(\mu_0, \overline{\mu})$ which induces the geodesic
\[
\mu_\tau : = (\pi_\tau)_\#\gamma
\]
between $\mu_0$ and $\overline{\mu}$. 

The time-rescaled geodesics 
\[\nu_{t} := \mu_{1-e^{-t}}\] 
between $\mu_0$ and $\overline{\mu}$ in a solution in the sense of Definition \ref{def:mdi} to the Cauchy problem
\[
\dot{\nu}_t \in (\pi_0, \pi_1 - \pi_0)_\#\Gamma_o(\nu_t, \overline{\mu})\,, \qquad \nu_0 = \mu_0\,.
\]  
Indeed, direct computations for $\varphi\in C_c^\infty(\Rd)$ leads to
\begin{align*}
\frac{d}{dt}\int_\Rd \varphi(x) \nu_t(dx)  &=\frac{d}{dt}\int_\Rd \varphi (x) \,\mu_{1-e^{-t}}(dx) \\
&=\frac{d}{dt}\int_{\Rd\times\Rd} \varphi(x+(y-x)(1-e^{-t}))\,\gamma(dx,dy)\\
&= \int_{\Rd\times\Rd} e^{-t} (y-x) \cdot\nabla \varphi(x+(y-x)(1-e^{-t}))\,\gamma(dx,dy)\,.
\end{align*}
Now, to push-forward $\gamma$ to a measure $\gamma_t\in \Gamma_o(\mu_{1-e^{-t}},\overline \mu)$, we perform the change of variable $(x,y)\mapsto (x+(y-x)(1-e^{-t}),y)$. Then, we apply the change of variable given by $(\pi_0, \pi_1 - \pi_0)$, so that 
\[(\pi_0, \pi_1 - \pi_0)_\#\gamma_t \in (\pi_0, \pi_1 - \pi_0)_\#\Gamma_o(\mu_{1-e^{-t}}, \overline{\mu})\]
to obtain
\begin{align*}
\frac{d}{dt}\int_\Rd \varphi(x) \nu_t(dx)  &= \int_{\Rd\times \Rd} (y- x)\cdot  \nabla \varphi (x) \,\gamma_t(dx,dy)\\
&= \int_{\T\Rd} v\cdot\nabla \varphi (x)\, (\pi_0, \pi_1 - \pi_0)_\#\gamma_t(dx,dv)\,.
\end{align*}

\section{Multi-agent consensus dynamics}
\label{Sec:DiffInclusion}
In the following, we start by formulating dynamics \eqref{eq:introMDE}, \eqref{eq:introMDI} as a MDI on the space $\mathcal{P}_2((\Rd)^N)$. Then, we present a time-discrete approximation scheme to construct solutions.

\subsection{MDI formulation of the dynamics}
We reformulate the problem of $N$ agents by considering an equation for the ensemble in $\mathcal{P}_2((\Rd)^N)$.
Recall $p_i: (\Rd)^N \to \Rd$ denote the projection on the $i$-th component.
For any $\mu \in \mathcal{P}((\Rd)^N)$, we will denote with $\mu^i = (p_i)_\#\mu$ their push-forward on each component.
Given a function $\lambda \in C_b(\mathcal{P}_{2}(\Rd), (0,\infty))$ the barycenter of the ensemble is computed via the weights $\lambda(\mu_t^i)$.
The case of an unweighed barycenter corresponds to the case $\lambda = 1$.

Let $\mathcal{V}: \mathcal{P}_2((\Rd)^N) \rightrightarrows \mathcal{P}_2(\T(\Rd)^N)$ be the MPVF given by
\begin{equation} \label{def:V}
\mathcal{V}[\mu]:= \bigotimes_{i=1}^N (\pi_0, \pi_1 - \pi_0)_\#  \Gamma_o\left(\mu^i, \Bar\left( \lambda \mathbb{P}_\mu\right)\right)\,,
\end{equation}
and we will refer to it throughout the rest of the paper as  $\mathcal{V}$.

\begin{proposition}
\label{p:uppersemi}
The MPVF $\mathcal{V}$ given by \eqref{def:V} is upper semi-continuous.
\end{proposition}

The proof is provided in Appendix \ref{appendix}. 
We are interested in studying the existence of solutions to the MDI 
\begin{equation}
 \dot{\mu}_t \in \mathcal{V}[\mu_t]\quad \mbox{a.e. }t\in[0,T],
\label{eq:MDI}
\end{equation}
supplemented by initial condition $\mu_0 = \otimes_{i=1}^N \mu_0^i$, $\mu_0^i \in{\mathcal{P}_2(\Rd)}$.
Recall by Definition \ref{def:mdi}, that a solution to \eqref{eq:MDI} is a curve $\mu\in \Lip([0,T],\mathcal{P}_2((\Rd)^N))$ such that for any $\phi \in C_c^\infty\left((\Rd)^N\right)$
\begin{equation}
\frac{d}{dt} \int_{(\Rd)^N} \phi(x) \mu_t(dx) \in \int _{(T\Rd)^N} (y-x)\cdot \nabla \phi(x) \mathcal{V}[\mu_t](dx,dv)\,.
\label{eq:MDI_weak}
\end{equation}

To underline the relation with the multi-agent system introduced in \eqref{eq:introMDI}, we establish in the Lemma below the equivalence between the above inclusion and the following system of inclusions 
\begin{equation}
\frac{d}{dt} \int_{\Rd} \psi(x) \mu_t^i(dx) =\int _{\Rd \times \Rd} (y-x)\cdot \nabla \psi(x) \,\gamma^i_t(dx,dy) \qquad i = 1, \dots,N
\label{eq:MDI_N_weak}
\end{equation}
for some $\gamma_t^i \in \Gamma_o\left(\mu_t^i, \Bar( \lambda \mathbb{P}_{\mu_t})\right)$,  and any $\psi\in C^\infty_c(\Rd)$.

\begin{lemma} 
\label{lemma:system2one}
Let $\mu_0 = \otimes_{i=1}^N \mu_0^i$, $\mu_0^i \in{\mathcal{P}_2(\Rd)}$.
Then, $\mu\in \Lip([0,T],\mathcal{P}_2((\Rd)^N))$ satisfies  \eqref{eq:MDI_weak} if and only if $\mu_t^i = (p_i)_\#\mu_t \in \Lip([0,T],\mathcal{P}_2(\Rd))$ satisfies \eqref{eq:MDI_N_weak} for all $i =1,\dots,N$.
\end{lemma}
\begin{proof}
Let $\{\mu^i_t\}_{i=1}^N$ satisfy \eqref{eq:MDI_N_weak} and $\gamma^i_t$ be a transport plan connecting the  $i$-th agent to the barycenter of the measures, towards which the agent is moving.
Consider $\mu_t = \otimes_{i=1}^N \mu_t^i$ and a test function $\phi\in C_c^{\infty}((\Rd)^N)$ of the form $\phi(x) = \prod_{i=1}^N\psi_i(p_i(x))$ with $\psi_i\in C^\infty_c(\Rd), i = 1,\dots,N$. 
It holds
\begin{align*}
\frac{d}{dt} \int_{(\Rd)^N} \prod_{i=1}^N &\psi_i(p_i(x)) \mu_t(dx) =\frac{d}{dt} \int_{\Rd} \prod_{i=1}^N \psi_i(x_i) \left( \mu^1_t(dx_1)\dots \mu^N_t(dx_N)\right)\\
&= \sum_{j=1}^N\left(\prod_{i\neq j}\int_{\Rd } \psi_i(x)  \mu^i_t(dx) \right)\frac{d}{dt}\int_{\Rd} \psi_j(x)  \mu^j_t(dx)\\
&= \sum_{j=1}^N \left(\prod_{i\neq j}\int_{\Rd } \psi_i(x)  \mu^i_t(dx) \right) \int _{\Rd\times \Rd }(y-x)\cdot  \nabla\psi_j(x)  \,\gamma^j_t(dx,dy)\\
&= \sum_{j=1}^N \left(\prod_{i\neq j}\int_{\Rd } \psi_i(x)  \mu^i_t(dx) \right) \int _{T\Rd }v\cdot  \nabla\psi_j(x)  \,(\pi_0, \pi_1 - \pi_0)_\#\gamma^j_t(dx,dv)\\
&= \sum_{j=1}^N \left(\prod_{i\neq j}\int_{\Rd } \psi_i(x)  \mu^i_t(dx) \right) \int _{\Rd  }X_{\mu^j_t}(x)\cdot  \nabla\psi_j(x)  \, \mu^j_t(dx)\\
&= \int_{(\Rd)^N }\left( X_{\mu^1_t}(p_1(x)),\dots, X_{\mu^N_t}(p_N(x))\right) \cdot  \nabla  \prod_{j=1}^N\psi_j(p_j(x))  \, \mu_t(dx)\\
&\in \int _{(T\Rd)^N } v \cdot  \nabla  \prod_{j=1}^N\psi_j(p_j(x))  \, \mathcal{V}[\mu_t](dx, dv).
\end{align*}
Therefore, \eqref{eq:MDI_weak} holds for $\phi(x) = \prod_{i=1}^N\psi_i(p_i(x))$.
The result can be extended to all $\phi \in C^\infty_c((\Rd)^N)$ 
thanks to the dense embedding in the weak topology \cite[Theorem 15.4]{duistermaat2010distribution}:
\begin{align*}
C^\infty_c(\Rd)^N &\to C^\infty_c\left((\Rd)^N\right)\, \\
 (\psi_1, \dots, \psi_N) &\mapsto \prod_{i=1}^N \psi_i\,.
\end{align*}

Let now $\mu_t\in \Lip \left([0,T],(\Rd)^N\right)$ satisfy the equation \eqref{eq:MDI_weak} on $(\Rd)^N$. Furthermore, let $V[\mu_t]\in\mathcal{V}[\mu_t]$, with $V[\mu_t]=\delta_{( X_{\mu^1_t}(x),\dots, X_{\mu^N_t}(x))}\otimes\mu_t$ and $\gamma^j_t=(\pi_0, \pi_1 - \pi_0)^{-1}_\# V[\mu_t]$, be the selection of PVF and optimal plans, respectively. Then, we have $\dot{\mu}_t=V[\mu_t]$.
As before, consider $\psi_i\in C^\infty_c(\Rd)$ with $ i = 1,\dots, N$,
and $\eta\in C^\infty_c((\Rd)^N)$ such that $\supp(\eta)\subset B(0,2)$, $\eta\equiv 1$ on $B(0,1)$, and $\|\nabla\eta\|_\infty\leq 1$. 
Define the sequence $\{\eta_n\}_{n\geq 1}$
\[
\eta_n(x)=\eta(x/n) \quad \textup{with} \quad \supp(\eta_n) \subset B(0,2n)
\]
such that $\lim_{n \to \infty} \eta_n = 1$ point-wise.
We compute
\begin{multline*}
\left|\int_0^T\int_{\T(\Rd)^N}(\psi_j\circ p_j)(x)(v \cdot\nabla \eta_n(x ))V[\mu_t](dx,dv)dt\right|\leq 
\\
T\|\psi_j\|_\infty\left(\int_{\T\left(B(0,2n)\setminus B(0,n)\right)} |v|^2 V[\mu_t](dx,dv)\right)^{1/2}.
\end{multline*}
We therefore obtain
\[
\lim_n\int_0^T\int_{\T(\Rd)^N}(\psi_j\circ p_j)(x)(v \cdot\nabla \eta_n(x ))V[\mu_t](dx,dv)dt=0\,.
\]
For any $s\in[0,T]$ we have
\begin{align*}
\int_{\Rd}\psi_j(x)&\mu^j_s(dx)-\int_{\Rd} \psi_j(x)\mu^j_0(dx) \\
&=\lim_{n\to\infty}\int_{(\Rd)^N} \eta_n(x)(\psi_j\circ p_j)(x)\mu_s(dx)-\int_{(\Rd)^N}\eta_n(x) ( \psi_j\circ p_j)(x) \mu_0(dx) \\
&=\lim_n\int_0^T\int_{\T(\Rd)^N}v \cdot\nabla(\eta_n(x)(\psi_j\circ p_j)(x))V[\mu_t](dx,dv)dt \\
&=\lim_n\int_0^s\Big(\int (\psi_j \circ p_j)(x)(v \cdot\nabla \eta_n(x ))V[\mu_t](dx,dv) \\
&\qquad + \int\eta_n(x)(v \cdot\nabla (\psi_j\circ p_j)(x))V[\mu_t](dx,dv)\Big)dt \\
&=\int_0^s\int_{\T(\Rd)^N}  v \cdot\nabla (\psi_j\circ p_j)(x)V[\mu_t](dx,dv)dt\\
&=\int_0^s\int_{\Rd}  v \cdot\nabla \psi_j (x)V[\mu^j_t](dx,dv)dt\,.
\end{align*}
This allows us to conclude and proves the equivalence of the two formulations \eqref{eq:MDI_weak} and \eqref{eq:MDI_N_weak}.
\end{proof}

\begin{remark}
Due to Lemma \ref{lemma:system2one}, the evolution described by the MDI \eqref{eq:MDI} does not depend on the particular initial conditions $\mu_0\in\mathcal{P}_2\left((\Rd)^N\right)$ but only on the marginals $\mu_0^1,\dots,\mu_0^N$.
For measures
$
\mu_0=\mu_0^1\otimes\dots\otimes\mu_0^N\in  \mathcal{P}_2\left((\Rd)^N\right)
$
we have solutions $\mu_t\in\Gamma(\mu_t^1,\dots,\mu_t^N)$ with no (optimal) coupling conditions required. 
\end{remark}

\subsection{Consensus dynamics for absolutely continuous agents}

To investigate existence of solutions to \eqref{eq:MDI}, we start by considering the case of absolutely continuous agents.
In particular we will considered absolutely continuous agents with bounded densities, that is,  $\mu_0^i \in \mathcal{P}_2^{a.c.}\left( \Rd \right)\cap L^\infty(\mathcal{L}^\d)$ for all $i = 1,\dots,N$.

\subsubsection{Iterative scheme}

We construct solutions to  \eqref{eq:MDI} starting by an explicit Euler approximation of the multi-agent dynamics. Given a fixed time horizon $T\in(0,\infty)$, let $\tau \in (0,1)$ denote the time discretization and $k\in \mathbb{N}$, $k\tau \leq T$ the corresponding iterations.
Starting from $\mu_{(0)}^i = \mu_{0}^i\in \mathcal{P}^{a.c.}_2(\Rd)\cap L^\infty(\mathcal{L}^\d)$ the agents' evolution is iteratively defined by the scheme
\begin{equation}
\begin{cases}
\overline \mu_{(k)} = \Bar(\lambda \mathbb{P}_{\mu_{(k)}}) \\
\mu_{(k+1)}^i  =  \left((1-\tau)\Id + \tau T_{\mu^i_{(k)}}^{\overline{\mu}_{(k)}}\right )_\#\mu_{(k)}^i \qquad i = 1,\dots,N\,.
\end{cases}
\label{eq:iterativeMDE}
\end{equation}
As before,  $\mu_{(k)} = \otimes_{i=1}^N \mu_{(k)}^i \in \mathcal{P}_2((\Rd)^N)$ denotes the collection of agents at time $k\tau \in [0,T]$.

The following lemma ensures well-posedness of the scheme and  important property regarding the agents' support.

\begin{lemma}\label{lem:SolutionsStayInsideCompact}
Assume the initial agents $\mu_{0}^i \in \mathcal{P}_{2}^{a.c.}(\Rd) \cap L^\infty(\mathcal{L}^\d)$, $ i= 1, \dots,N$ to be compactly supported. For a given time step $\tau \in (0,1)$, time horizon $T>0$, and weight function $\lambda \in C_b(\mathcal{P}_{2}(\Rd), (0,\infty))$, the sequence generated by the scheme \eqref{eq:iterativeMDE} satisfies $\mu_{(k)}^i \in \mathcal{P}_{2}^{a.c.}( \Rd ) \cap L^\infty(\mathcal{L}^{\d})$ and
\[\textup{Conv} \left(  \bigcup_{i=1}^N \textup{supp}(\mu_{(k)}^i)\right) \subseteq
\textup{Conv} \left(  \bigcup_{i=1}^N \textup{supp}(\mu_{0}^i)\right)\]
for all $k\tau \in [0,T]$, 
where $\textup{Conv}(A)$ is the convex hull of the set $A\subseteq\Rd$.
\end{lemma}


\begin{proof} 
Set $\Omega:= \textup{Conv} \left(  \bigcup_{i=1}^N \textup{supp}(\mu_{0}^i)\right)$ 
and assume the claim holds at the step $k$, that is, $\mu_{(k)}^i \in \mathcal{P}_{2}^{a.c.}(\Omega) \cap L^\infty(\mathcal{L}^{\d})$.

Due to the assumptions, the weights $\lambda(\mu_{(k)}^i)$ are positive, and so are the normalized ones given by  $ \tilde \lambda_{(k)}^i = \lambda(\mu_{(k)}^i)/\sum_j \lambda(\mu_{(k)}^j)$.
Since the barycenter definition does not depend on the normalization constant,  $\overline{\mu}_{(k)}$ is also the barycenter among $\{\mu_{(k)}^i\}_{i=1}^N \subset L^{\infty}(\mathcal{L}^\d)$ with positive weights $\tilde \lambda^i_{(k)}, \sum_{i=1}^N\tilde \lambda_{(k)}^i = 1$. By applying the regularity result for the barycenter \cite[Theorem 5.1]{carlier2011bary}, we obtain 
\[
\|\overline{\mu}_{(k)}^i\|_{L^\infty} \leq \inf_{i = 1, \dots,N} \frac1{\tilde{\lambda}_{(k)}^i} \| {\mu}_{(k)}^i\|_{L^\infty}\,.
\]
Additionally, it holds, again from \cite{carlier2011bary},
\[\textup{supp}(\overline{\mu}_{(k)}) \subseteq \sum_{i=1}^N\tilde{ \lambda}_{(k)}^i\,\textup{supp}(\mu_{(k)}^i) \subseteq \Omega\,,\] 
and we  deduce $\overline{\mu}_{(k)} \in \mathcal{P}_{2}^{a.c.}\left( \Omega \right) \cap L^\infty(\mathcal{L}^{\d})$.

The new agent $\mu^i_{(k+1)}$ consists of the McCann’s interpolation between the old agent position $\mu_{(k)}^i$ and the barycenter $\overline{\mu}_{(k)}$. Since the McCann's interpolation solves the barycenter problem with weights $(1-\tau)$ and $\tau$, see \cite[Section 6.2]{carlier2011bary}, from $\overline{\mu}_{(k)}, \mu_{(k)}^i \in \mathcal{P}_{2}^{a.c.}\left( \Omega \right) \cap L^\infty(\mathcal{L}^{ \d})$,  we have $\supp(\mu^i_{(k+1)}) \subseteq \Omega$ and the inequality
\[
\|\mu^i_{(k)}\|_{L^\infty}\leq \frac{\|\mu^i_{0}\|_{L^\infty}}{(1-\tau)^k}\,.
\]
Thanks to the bound $k\tau\leq T$, we have
\[(1-\tau)^{-k} \leq (1-\tau)^{-T/\tau} = e^{T} + o(1) \quad \text{as} \quad \tau \to 0\,.
\]
Therefore, for $\tau$ sufficiently small, we have
\[
\|\mu^i_{(k)}\|_{L^\infty}\leq( e^T +1)\|\mu^i_{0}\|_{L^\infty} 
\]
and we can conclude $\mu^i_{(k+1)}\in \mathcal{P}_{2}^{a.c.}\left( \Omega \right) \cap L^\infty(\mathcal{L}^{\d})$ for all $i = 1, \dots,N$ and with uniform bounds on the densities. 
\end{proof}

\subsubsection{Existence of solutions}

We show  that the iterative scheme \eqref{eq:iterativeMDE} allows to construct solutions to the multi-agent consensus dynamics \eqref{eq:MDI}.

\begin{theorem}
\label{t:ac}
Assume $\lambda \in C_b(\mathcal{P}_{2}(\Rd), (0,\infty))$ and consider a time horizon $T\in(0,\infty)$.
Given an initial datum $\mu_{0} = \otimes_{i=1}^N\mu_{0}^i$ such that $\mu_{0}^i  \in  \mathcal{P}_{2}^{a.c.}(\Rd) \cap L^\infty(\mathcal{L}^\d)$ are compactly supported, then there exists a solution $\mu \in \Lip\left([0,T], \mathcal{P}_2^{a.c.}((\Rd)^N)\right)$ to the MDI \eqref{eq:MDI} with $\lim_{t\to 0}\mu_t = \mu_0$.
\end{theorem}

\begin{proof} We will divide the proof in three steps. First, we interpolate the time-discrete agents constructed by the iterative scheme \eqref{eq:iterativeMDE}. Then, we show existence of a limit as $\tau \to 0$, $k\to\infty$, and finally we show the limit solves the desired MDI.

\textbf{Step 1.} As before, we set $\Omega:= \textup{Conv} \left(  \bigcup_{i=1}^N \textup{supp}(\mu_{0}^i)\right)$. 
To stress the dependence on $\tau$, we indicate with $\mu^{i,\tau}_{(k)} \in \mathcal{P}^{a.c.}_2(\Omega) \cap L^\infty(\mathcal{L}^{\d})$ and $\mu^\tau_{(k)} = \otimes_{i=1}^N \mu_{(k)}^{i,\tau}$ the measures constructed according to \eqref{eq:iterativeMDE} with parameter $\tau\in(0,1)$. The fact that they are supported over $\Omega$ follows from Lemma \ref{lem:SolutionsStayInsideCompact}. As the agents' instantaneous movement is along geodesics, for any $k$ it holds
\begin{align*}
W_2(\mu_{(k)}^{i,\tau}, \mu_{(k+1)}^{i,\tau}) \leq \tau W_2(\mu_{(k)}^{i,\tau}, \overline{\mu}_{(k)}^{\tau}) \leq \tau \diam(\Omega)\, .
\end{align*}
Since the product of optimal transport couplings between $\mu^{i,\tau}_{(k)}, \, \mu^{i,\tau}_{(k+1)}$, $i = 1,\dots,N$ produces a sub-optimal coupling between $\mu_{(k)}^\tau$ and $\mu_{(k+1)}^\tau$, it holds
\begin{equation}
W_2(\mu^{\tau}_{(k)}, \mu^{\tau}_{(k+1)})^2
\leq \sum_{i=1}^N W_2(\mu^{i,\tau}_{(k)}, \mu^{i,\tau}_{(k+1)})^2 \leq N \tau^2 \diam(\Omega)^2.
\label{eq:lipconst}
\end{equation}
Now, let $\mu^{i,\tau}_t$, $t \in [0,T]$ be the piecewise geodesic interpolation of $\mu^{i,\tau}_{(k)}$, that is,
\[
\mu_t^{i,\tau}  :=  \left((1- t +k\tau)\Id + (t - k\tau) T_{\mu^i_{(k)}}^{\overline{\mu}_{(k)}}\right )_\#\mu_{(k)}^{i,\tau} \qquad \textup{for}\quad t \in [k\tau, (k+1)\tau],
\]
and $\mu^\tau_t = \otimes_{i=1}^N \mu_t^{i,\tau}$. 

\textbf{Step 2.}
Thanks to \eqref{eq:lipconst}, we have $\mu^\tau \in \Lip\left([0,T], \mathcal{P}_2^{a.c.}(\Omega^N)\right)$ with Lipschitz constant $\sqrt{N}\diam(\Omega)$ and $\supp(\mu^\tau)\subseteq\Omega$ thanks to Lemma \ref{lem:SolutionsStayInsideCompact}.
Since the domain $\Omega$ is compact, by a generalized version of Ascoli--Arzelà theorem for metric spaces (see \cite[Theorem 18]{kelley2017general}), there exists a limit curve to $\mu^\tau$ as $\tau \to 0$ in $\Lip([0,T], \mathcal{P}_2(\Omega^N))$.
Consider a sequence $(\tau_j)_{j \in \mathbb{N}}$ with $\lim_{j\to\infty}\tau_j=0$, such that $\mu^{\tau_j}$ converges uniformly to $\mu\in \Lip\left([0,T], \mathcal{P}_2(\Omega^N)\right)$. We will show that $\mu_t$ admits a density. For any $t\in [0,T]$ and $\varphi\in C((\Rd)^N)$, it holds
\[
\lim_{j \to \infty} \int_{(\Rd)^N} \varphi(x)\,\mu^{\tau_j}_t(dx) = \int_{(\Rd)^N}\varphi(x)\, \mu_t(dx)\,.
\]
Let $f_t^{\tau_j} \in L^\infty(\Omega^N)$ be density of $\mu^{\tau_j}_t$ with respect to $\Leb^{\d N}$. 
Consider a test function $\varphi \in C_c^\infty((\Rd)^N)$, and for any $\ve>0$ let $g \in  L^1(\Omega^N)$ such that $\| g - \varphi\|_{L^1} \leq \ve$. It holds for any $j,l\in\N$
\begin{align*}
\left|\int_{(\Rd)^N} g(x) (f^{\tau_j}_t (x)-f^{\tau_l}_t (x)) dx\right|&\leq \left|\int_{(\Rd)^N} \varphi(x) (f^{\tau_j}_t (x)-f^{\tau_l}_t (x) )dx\right|\\
& \qquad+\varepsilon\sup\{\|f^{\tau_j}_t\|_\infty,\|f^{\tau_l}_t\|_\infty\}\\
&\leq\|\nabla\varphi\|_{L^\infty} W_2\left(\mu^{\tau_j}_t ,\mu^{\tau_l}_t\right)+\varepsilon\sup\{\|f^{\tau_j}_t\|_\infty,\|f^{\tau_l}_t\|_\infty\}\,.
\end{align*}
Thanks to the fact that $\|f^{\tau_l}_t\|_\infty$ are uniformly bounded, and thanks to Lemma \ref{lem:SolutionsStayInsideCompact}, the above shows that $\int g f^{\tau_j}_t \Leb^{\d N} $ has a limit. Since the choice of $g$ and $\varphi$ were general, we obtain that $f_t^{\tau_j}$ has a weak-$\star$ limit $f_t\in L^\infty(\Omega^N)$, which in particular is the density of $\mu_t$ and $\mu_t = f_t \Leb^{\d N}$.

\textbf{Step 3.} We show that $\mu$ satisfies \eqref{eq:MDI_weak}: we note that, since $\mathcal{V}$ is upper semi-continuous on the space of measures (Proposition \ref{p:uppersemi}), it is continuous on $\mathcal{P}^{a.c.}_2(\Rd)$ because it is single valued on that space.
Let $K^\tau\in\N$ be the largest integer such that $K^\tau \tau\leq T$ and $T_{(k)}^\tau$ be the collection of all transport maps 
\[T_{(k)}^\tau:= \left (T_{\mu_{(k)}^1}^{\overline{\mu}_{(k)}}, \dots, T_{\mu_{(k)}^N}^{\overline{\mu}_{(k)}}\right)^\top\,,
\]
it holds
\begin{align*}
\langle\varphi,\mu^\tau_{T}-\mu^\tau_0\rangle&=\sum_{k=0}^{K^\tau-1}\langle \varphi, \mu^\tau_{(k+1)}\rangle - \langle \varphi, \mu^\tau_{(k)} \rangle\\
&=\sum_{k=0}^{K^\tau-1}\int_{(\Rd)^N} \left(\varphi \left(x+\tau (T^\tau_{(k)}(x)-x)\right) - \varphi (x) \right) \mu^\tau_{(k)} (dx)\,.
\end{align*}
By applying Taylor's theorem on $\varphi$ for all $k = 0,\dots,K^\tau-1$ we obtain functions $h_k$
\begin{align*}
 \int\varphi\left(x+(T^\tau_{(k)}(x)-x)\tau\right)& \mu^\tau_{(k)}(dx)-\int\varphi(x) \mu^\tau_{(k)}(dx) \\
 &= \tau\int \left(T^\tau_{(k)}(x)-x\right)\cdot\nabla\varphi(x) \mu^\tau_{(k)}(dx)+h_k(\tau)\frac{\tau^2}{2} \\
&= \tau\left\langle\nabla\varphi , \mathcal {V}\left[\mu^\tau_{(k)}\right]\right\rangle + \frac{\tau^2}{2} h_k(\tau)\,,
\end{align*}
with a remainder term $| h_k| \leq \| \nabla^2 \varphi\|_{L^\infty}N \diam(\Omega)^2$, thanks to $|T^\tau_{(k)}(x)-x|^2 \leq N  \diam(\Omega)^2$.
By summing all $k = 0,\dots, K^\tau-1$ and using $\tau K^\tau \leq T$, we get
\begin{align}
\left| \langle \varphi, \mu^\tau_{(K^\tau)}-\mu^\tau_{(0)}\rangle -\sum_{k=0}^{K^\tau-1}\tau\left\langle \nabla \varphi , \mathcal {V}\left[\mu^\tau_{(k)}\right]\right\rangle\right|&\leq \tau T\|\nabla^2\varphi\|_{L^\infty} N \diam(\Omega)^2 \notag\\
&=: C_1  \tau \,.
\label{eq:bound1}
\end{align}

Consider again the sequence $\tau_j\to 0$ such that $\lim_{j\to\infty}W_2\left(\mu^{\tau_j}_t,\mu_t\right)=0$ uniformly.
The closed set
\[
A=\Cl\Bigg(\bigcup_{j\in\N}\mu^{\tau_j}([0,T])\Bigg)
\]
is, then, compact. Since all the measures inside $A$ are absolutely continuous we have that $\mathcal{V}$ is single valued on all $A$ and, therefore, it is uniformly continuous. In particular, for any $\varepsilon>0$, there exists a $\delta>0$ such that for all $\mu,\nu\in A$: $W_2(\mu,\nu)<\delta$ implies $W_2(\mathcal{V}[\mu],\mathcal{V}[\nu])<\varepsilon$.
Since $W_2(\mu^{\tau_j}_s,\mu^{\tau_j}_t)\leq |s-t|\diam(\Omega)\sqrt{N}$, we also get that for any $\tau_j<\delta/(\diam(\Omega)\sqrt{N})$ it holds
\[
W_2(\mathcal{V}[\mu^{\tau_j}_s],\mathcal{V}[\mu^{\tau_j}_t)]<\varepsilon\quad\textup{if} \quad |s-t|<\tau_j.
\]
The above estimate leads to
\begin{align}
\Bigg | \sum_{k=0}^{K^{\tau_j}-1} \! \tau_j \langle \nabla \varphi,\mathcal {V}[\mu^{\tau_j}_{(k)}]\rangle &\!-\! \int_0^T\!\!\!\langle\nabla\varphi,\mathcal {V}[\mu^{\tau_j}_t]\rangle dt \Bigg | 
=\left|\sum_{k=0}^{K^{\tau_j}-1}\!\!\int_{k{\tau_j}}^{(k+1){\tau_j}}\!\!\langle \nabla\varphi,\mathcal {V}[\mu^{\tau_j}_{(k)}]-\mathcal {V}[\mu^{\tau_j}_t]\rangle dt \right|
\notag
\\
&\leq \sum_{k=0}^{K^{\tau_j}-1}{\tau_j}\|\nabla\varphi\|_{W^{1,\infty}}\diam(\Omega)\sup_{t\in [k{\tau_j},(k+1){\tau_j}]} W_2(\mathcal {V}[\mu^{\tau_j}_{(k)}],\mathcal {V}[\mu^{\tau_j}_t]) \notag
\\
&\leq  T\|\nabla\varphi\|_{W^{1,\infty}}\diam(\Omega) \varepsilon =: C_2 \varepsilon\,.
\label{eq:bound2}
\end{align}
Thanks again to the uniform continuity of $\mathcal{V}$ on the set $A$, we can take $\tau_j$ sufficiently small such that $\sup_{t\in[0,T]}W_2(\mu_t,\mu^{\tau_j}_t)\leq\delta$ implies $\sup_{t\in[0,T]}W_2(\mathcal {V}[\mu_t],\mathcal {V}[\mu_t^{\tau_j}])\leq\varepsilon$, leading to 
\begin{align}
\left|\left\langle\varphi,\mu_T-\mu_0\right\rangle - \left\langle\varphi,\mu^{\tau_j}_T-\mu^{\tau_j}_0\right\rangle\right|&\leq 2 \|\nabla\varphi\|_{L^\infty} \sup_{t\in [0,T]} W_2(\mu_t,\mu_t^{\tau_j})
\leq 2 \|\nabla\varphi\|_\infty\delta \notag \\
& =: C_0 \delta
\label{eq:bound0}
\end{align}
and
\begin{align}
\left|\int_0^T\left\langle\varphi , \mathcal {V}[\mu_{t}]-\mathcal {V}[\mu^{\tau_j}_{t}]\right\rangle dt \right|&\leq T\|\nabla\varphi\|_{W^{1,\infty}}\diam(\Omega)
\sup_{t\in[0,T]}W_2(\mathcal{V}[\mu_t^{\tau_j}],\mathcal{V}[\mu_t]) \notag
\\
&\leq T\|\nabla\varphi\|_{W^{1,\infty}}\diam(\Omega)\varepsilon=: C_3 \varepsilon\,.
\label{eq:bound3}
\end{align}
We put now all the above inequalities together to get the inequality
\begin{align*}
\left|\left\langle\varphi,\mu_T-\mu_0\right\rangle-\int_0^T\left\langle\varphi , \mathcal {V}[\mu_{t}]\right\rangle\right|&\leq 2 C_0\delta+C_1\tau_j+C_2\varepsilon+C_3\varepsilon
\end{align*}
with constants $C_i>0$, $i = 0,1,2,3$ depending on $N$, $\diam(\Omega)$, $T$ and on $\|\nabla\varphi\|_{W^{1,\infty}}$.
Since $\varepsilon>0$ is arbitrary, $\delta$ can be chosen to be smaller than $\varepsilon>0$ and from the choice $\tau_j < \delta /(\diam(\Omega) \sqrt{N})$, we can conclude that
\[
\left|\left\langle\varphi,\mu_T-\mu_0\right\rangle-\int_0^T\left\langle\varphi , \mathcal {V}[\mu_{t}]\right\rangle\right|=0\quad\textup{for all}\quad\varphi\in C^\infty_c((\Rd)^N).
\]
\end{proof}

\subsection{The set of solutions is closed}

Before extending the existence result Theorem \ref{t:ac} to general agents with compact support in $\mathcal{P}_2(\Rd)$, we show that the solutions to \eqref{eq:MDI}, that is, 
\[
\Sol[\mathcal{V}] := \left\{\mu \in C([0,T], \mathcal{P}(\RR^{\d N}))
\;:\; \mu\;\; \textup{satisfies \eqref{eq:MDI} with}\;  \mu_0 \in \mathcal{P}_2(\RR^{\d N})\right \},
\]
forms a closed set.
To this end, we reformulate both the MDI \eqref{eq:MDI} and the upper semi-continuity property of Definition \ref{def:uppersemi} by an asymmetric Hausdorff quasimetric. This is defined for any compact subsets $A,B$ of $\mathcal{P}_2(\RR^{\d N})$ and given by 
\[
\mathrm{d}^{W_2}_H(A,B):=\sup_{\mu\in A}
\inf_{\nu \in B} W_2(\mu, \nu)\,.
\]

\begin{lemma}\label{Prop:USC1}
A MPVF $\mathcal{W}:\mathcal{P}_2((\Rd)^N)\rightrightarrows\mathcal{P}_2(\T(\Rd)^N)$ is upper semi-continuous if and only if for any $\nu \in \mathcal{P}_2((\Rd)^N)$ it holds
\begin{align}\label{eq:SetUSC}
\lim_{\mu\to\nu} \mathrm{d}^{W_2}_H\left( \mathcal {W}[\mu] , \mathcal {V}[\nu]\right)=0\,.
\end{align}
\end{lemma}
\begin{proof}
Given $\mathcal{W}$ upper semi-continuous we proceed by contradiction. Assume the existence of a converging sequence $\mu^n\to\mu$ in $\mathcal{P}_2(\RR^{\d N})$ and $\alpha>0$ such that
\[
\lim_{n\to\infty} \mathrm{d}^{W_2}_H\left( \mathcal {W}[\mu^n] , \mathcal {W}[\nu]\right)>\alpha\, .
\]
By  the definition of the asymmetric Hausdorff quasimetric $\d_H^{W_2}$ we get the existence of $\varepsilon>0$ small enough and $N_\varepsilon\in\N$ large enough such that for any $n>N_\varepsilon$ 
\begin{align*}
\alpha<\mathrm{d}^{W_2}_H\left( \mathcal {W}[\mu^n] , \mathcal {W}[\mu]\right)+\varepsilon =\sup_{\rho\in \mathcal {W}[\mu^n] }\inf_{\sigma\in \mathcal {W}[\mu] }W_2(\rho,\sigma)+\varepsilon\,.
\end{align*}
Hence, there exists a sequence $\rho_n\in\mathcal{W}[\mu^n]$ such that for $n\in\N$ large enough,
\begin{align*}
\alpha<W_2(\rho_n,\sigma)\qquad\textup{for all} \quad \sigma\in\mathcal{W}[\mu]\, ,
\end{align*}
which is a contradiction as $\mathcal{W}$ is upper semi-continuous.
Now, if we assume that $\mathcal{W}$ satisfies \eqref{eq:SetUSC} then by  definition of $d_{H}^{W_2}$ we directly obtain the upper semi-continuity of $\mathcal{W}$.
\end{proof}

Since we are  interested in testing the distributions contained inside $\mathcal{V}[\mu]$, we introduce, as a second step, upper semicontinuity for $\langle\nabla\varphi,\mathcal{V}[\mu]\rangle$.
We define the Hausdorff quasimetric on the subsets of the reals, so for $A,B \subset \RR$ we have
\[
\mathrm{d}^{\RR}_H(A,B):=\sup_{x\in A}
\inf_{y \in B} |x - y|\,.
\]
We can already note that the differential inclusion \eqref{eq:MDI} can be directly rewritten as
\[
\mathrm{d}^\RR_H\left(\langle\varphi,\mu_T-\mu_0\rangle,\int_0^T\langle\nabla\varphi,\mathcal {V}[\mu_t]\rangle dt\right)=0\qquad \textup{for all} \quad \varphi\in C^\infty_c( \RR^{\d N})\, .
\]
The following Lemma motivates this reformulation.

\begin{lemma}\label{Prop:USC2}
A MPVF $\mathcal{W}:\mathcal{P}_2((\Rd)^N)\rightrightarrows\mathcal{P}_2(\T(\Rd)^N)$ is upper semi-continuous if and only if for any $\nu \in \mathcal{P}_2((\Rd)^N)$ and $\varphi\in C^\infty_c(\RR^{\d N})$ it holds
\begin{align}\label{eq:AltUSC}
\lim_{\mu\to\nu} \mathrm{d}^\RR_H\left(\langle\nabla\varphi,\mathcal {W}[\mu]\rangle,\langle\nabla\varphi,\mathcal {W}[\nu]\rangle\right)=0\, .
\end{align}
\end{lemma}

\begin{proof}

Assume $\mathcal{W}$ to be upper semi-continuous.
Given Lemma \ref{Prop:USC1} we have that for any sequence $\mu^n\in\mathcal{P}_2(\Rd)$ converging to $\mu\in\mathcal{P}_2(\Rd)$, and $\varphi\in C^\infty_c(\Rd)$ there exists $C>0$ and it holds
\[
\mathrm{d}^\RR_H\left(\langle\nabla\varphi,\mathcal {W}[\mu^n]\rangle,\langle\nabla\varphi,\mathcal {W}[\nu]\rangle\right)\leq C \|\nabla\varphi\|_{W^{1,\infty}} \, \mathrm{d}^{W_2}_H\left( \mathcal {W}[\mu^n]\rangle,\mathcal {W}[\nu]\right)\,,
\]
which leads to \eqref{eq:AltUSC}.

Now, assume that $\mathcal{W}$ satisfies \eqref{eq:AltUSC} and prove by  contradiction. This grants us the existence of converging sequences $\mu^n\to\mu$ in $\mathcal{P}_2((\Rd)^N)$ and $\sigma_n\to\sigma$ in $\mathcal{P}_2(T(\Rd)^N)$ such that $\sigma_n\in\mathcal{W}[\mu^n]$ but $\sigma\not\in\mathcal{W}[\mu]$.
We consider now $\alpha:=$$\mathrm{d}^{W_2}_H(\sigma,\mathcal{W}[\mu])$ ${>0}$, fix $0<\varepsilon<\alpha/2$, and choose $\varphi\in C^\infty_c(\RR^{\d N})$ such that
\[
\inf_{\xi\in\mathcal {V}[\mu]}\left|\langle\nabla\varphi,\sigma\rangle-\langle\nabla\varphi,\xi\rangle\right|>\alpha-\varepsilon\,.
\]
By the lower semi-continuity of the above we can find $n\in \mathbb{N}$ large enough such that 
\begin{align*}
\mathrm{d}^\RR_H\left(\langle\nabla\varphi,\mathcal {W}[\mu^n]\rangle,\langle\nabla\varphi,\mathcal {W}[\mu]\rangle\right)&=\sup_{\rho\in \mathcal {W}[\mu^n]}\inf_{\xi\in\mathcal {W}[\mu]}\left|\langle\nabla\varphi,\rho\rangle-\langle\nabla\varphi,\xi\rangle\right|\\
&\geq \inf_{\xi\in\mathcal {W}[\mu]}\left|\langle\nabla\varphi,\sigma_n\rangle-\langle\nabla\varphi,\xi\rangle\right|\\
&\geq \alpha-\varepsilon\,.
\end{align*}
For sufficiently small $\ve>0$ we obtain a contradiction to $\mathcal{W}$ being uppersemicontinuous.
\end{proof}

\begin{theorem}\label{Thm::Closure}
The  set $\Sol[\mathcal{V}]$ is closed.
\end{theorem}

\begin{proof}
Recall $\mathcal{V}$ is an upper semi-continuous MPVF with compact values,
and consider a converging sequence of curves of measures $\mu^n\to\mu$, $\mu^n\in\Sol[\mathcal{V}]$.
We fix any $\varphi\in C^\infty_c(\RR^{\d N})$, and introduce the function $R:\mathcal{P}_2(T\RR^{\d N})\to\RR$ defined as 
\begin{align*}
R(\sigma):= \mathrm{d}^\RR_{H}(\langle \nabla\varphi, \sigma\rangle,\langle \nabla\varphi, \mathcal {V}[\mu]\rangle)\,
\end{align*}
which is continuous because it is the distance from a closed set.
Since the range of all the $\mu^n$ and $\mu$ is a compact set, the set
\[
\mathcal{V}\left(\Cl\left(\bigcup_{n\in \mathbb{N}}\mu^n([0,T]) \right)\right)
\]
is compact because it is closed thanks to $\mathcal{V}$ being upper semi-continuous and has values in a compact set as shown in \cite{aubin1984diffinclusions}. As a consequence, $R$ is uniformly continuous.
For any $\varepsilon_1>0$ we can find $N_1\in\N$ such that  $\sup_{t\in[0,T]}W_2(\mu^n_t,\mu_t)<\varepsilon_1$ for all $n>N_1$.
Since $\mu^n$ are solutions to the MDI \eqref{eq:MDI}, let us select $\sigma_t^n\in\mathcal{V}[\mu_t^n]$ such that
\[
\langle \varphi,\mu^n_T-\mu^n_0\rangle=\int_0^T\langle\nabla\varphi,\sigma^n_t\rangle\, dt\,,
\]
leading to $\lim_{n\to\infty}R(\sigma^n_t)=0$ for almost every $t\in[0,T]$.
Now, for any $\varepsilon_2>0$ we can find $N_2\in\N$ such that $R(\sigma^n_t)<\varepsilon_2$ for all $n>N_2$ and almost every time $t\in[0,T]$.
By choosing $n>\max\{N_1,N_2\}$ we can write
\begin{align*}
\mathrm{d}^\RR_H\bigg(\langle \varphi,\mu_T-\mu_0\rangle, &\int_0^T\langle\nabla\varphi,\mathcal {V}[\mu_t]\rangle dt\bigg)=\inf_{\sigma_t\in\mathcal{V}[\mu_t]}\left|\langle \varphi,\mu_T-\mu_0\rangle-\int_0^T\langle\nabla\varphi,\sigma_t\rangle dt \right| \\
&\leq \left|\langle \varphi,(\mu_T-\mu_0)-(\mu^n_T-\mu^n_0)\rangle\right|+\inf_{\sigma_t\in\mathcal{V}[\mu_t]}\left|\int_0^T\langle\nabla\varphi,\sigma^n_t-\sigma_t\rangle dt \right| \\
&\leq 2\|\nabla\varphi\|_{\infty}\varepsilon_1+\inf_{\sigma_t\in\mathcal{V}[\mu_t]}\left|\int_0^T\langle\nabla\varphi,\sigma^n_t-\sigma_t\rangle dt \right| \\
&\leq 2\|\nabla\varphi\|_{\infty}\varepsilon_1+T\varepsilon_2
\end{align*}
and conclude that $\mu\in\Sol[\mathcal{V}]$. This proves that the set of solutions $\Sol[\mathcal{V}]$ is closed. 
\end{proof}

\section{Main result}
\label{Sec:mainresult}
In the previous section, we used the time-discrete scheme \eqref{eq:iterativeMDE} to construct solutions to the MDI \eqref{eq:MDI} with initial data given by compactly supported agents $\mu^i_0 \in \mathcal{P}_{2}^{a.c.}(\Rd) \cap L^\infty(\mathcal{L}^\d)$, $i = 1, \dots,N$ (Theorem \ref{t:ac}). 
We extend the existence result to the general case  $\mu^i_0 \in \mathcal{P}_2(\Rd)$, $i=1,\dots,N$ keeping the assumption on the compact support. We perform a mollification procedure and exploit the closeness property of the solution set established in Theorem \ref{Thm::Closure}.

\begin{theorem}
\label{t:mdi}
Assume $\lambda \in C_b(\mathcal{P}_{2}(\Rd), (0,\infty))$ and consider a time horizon $T\in(0,\infty)$.
Given an initial datum $\mu_{0} = \otimes_{i=1}^N\mu_{0}^i$ such that $\mu_{0}^i  \in  \mathcal{P}_{2}(\Rd)$ are compactly supported, then there exists a solution $\mu \in \Lip \left([0,T], \mathcal{P}_2((\Rd)^N)\right)$ to the MDI \eqref{eq:MDI} with $\lim_{t\to 0}\mu_t = \mu_0$.
\end{theorem}
\begin{proof}
In the following, for any $\nu\in\mathcal{P}(  \Omega ^{  N})$, we denote with $\Sol[\mathcal{V}](\nu)\subset \Sol[\mathcal{V}]$ the set of solutions to \eqref{eq:MDI} with initial data $\nu$. Thanks to Theorems \ref{t:ac} and \ref{Thm::Closure}, and the fact that Wasserstein spaces are separable, the set $\Sol[\mathcal{V}](\nu)$ is upper semi-continuous.

We define, as before, $\Omega:= \textup{Conv} \left(  \bigcup_{i=1}^N \textup{supp}(\mu_{0}^i)\right)$ and introduce for any $n\in\mathbb{N}$ the neighbourhoods $\Omega_{1/n}=\cup_{x\in\Omega}B(x,1/n)$. Consider a family of mollifiers $\varphi_{n}\in C^{\infty}_c((\Rd)^N)$ with support in $\Omega_{1/n}$ such that we have $\varphi_n\star\mu_0\in\mathcal{P}_2(\Omega^N_{1/n})$ and
\[
W_2(\varphi_n\star\mu_0, \mu_0) \to 0
\]
Since we have existence of solutions for absolutely continuous case (Theorem \ref{t:ac}), there exists a sequence of solutions
\[
\mu^n\in\Sol[\mathcal {V}](\varphi_n\star\mu_0)\,.
\]
The sequence admits a limit point since it is supported on a compact set and $\Lip(\mu^n)\leq \sqrt{N}\diam(\Omega_{1/n})\leq \sqrt{N}(\diam(\Omega)+1)$.
Therefore, it admits a convergent subsequence thanks to Ascoli--Arzel\`a, and we have that the limit $\mu$ is a solution to \eqref{eq:MDI} with $\lim_{t\to 0}\mu_t = \mu_0$ thanks to Theorem \ref{Thm::Closure}.
\end{proof}

\begin{remark}
We underline that our existence result is constructive (via the iterative scheme \eqref{eq:iterativeMDE}) only for the case of absolutely continuous agents. For numerical purposes it may be interesting to consider directly a general time discretization of the MDI without any assumption on the initial data. Starting from $\mu^i_{(0)} = \mu^i_0 \in \mathcal{P}_2(\Rd)$, a simple explicit Euler discretization  with time step $\tau\in(0,1)$ which generalizes \eqref{eq:iterativeMDE} is given by
\begin{equation}
\begin{cases}
\overline \mu_{(k)} \in \Bar(\lambda \mathbb{P}_{\mu_{(k)}}) \\
\mu_{(k+1)}^i  =   (\pi_{\tau})_\# \gamma_{(k)}^{i} & \textup{for some} \;\; \gamma_{(k)}^{i} \in \Gamma_o \left(\mu_{(k)}^{i}, \overline \mu_{(k)}  \right) \qquad i = 1,\dots,N\,.
\end{cases}
\label{eq:iterativeMDI}
\end{equation}
where $\pi_\tau(x,y)  = (1-\tau)x + \tau y$, and $\mu_{(k)} = \otimes_{i=1}^N \mu_{(k)}^i \in \mathcal{P}_2((\Rd)^N)$ denotes the collection of agents. Whether \eqref{eq:iterativeMDI} converges to solutions to \eqref{eq:MDI} is left for future work.

\end{remark}

\begin{remark}
Consider the consensus model for agents in $\Rd$ evolving according to the system of ODEs 
\begin{equation}
\frac{d}{dt} x^i_t =  \frac1{N} \sum_{i=1}^N (x_t^j -x^i_t ) \quad i = 1,\dots,N \,,
\label{eq:odeconsensus}
\end{equation}
which is a simplified, time-continuous version of the Kuramoto \cite{kuramoto1984} and Hegselmann--Krause \cite{Hegselmann2002} models. Each agent moves towards the unweighted mean $\overline{x}_t := 1/N\sum_i x_t^i$.
Formally, we evaluate \eqref{eq:MDI_N_weak} at $\mu^i=\delta_{x^i}$, then the system \eqref{eq:MDI_N_weak} reads
\[
\frac{d}{dt}\varphi(x^i_t)=\nabla\varphi(x^i_t)\cdot(\Bar(\lambda\mathbb{P}_{\mu_t})-x^i_t)\qquad\textup{for}\quad i =  1,\dots,N
\]
that simplifies to
\[
\frac{d}{dt} x^i_t =  \Bar(\lambda\mathbb{P}_{\mu_t}) -x^i_t\qquad \textup{for}\quad i =  1,\dots,N\,.
\]

For $\lambda = 1/N$, it holds $\Bar(\lambda\mathbb{P}_{\mu_t}) = \overline{x}_t$. Formally, solutions to \eqref{eq:odeconsensus} are solutions to the MDI \eqref{eq:MDI} when agents are Dirac deltas. The same can be shown to hold in presence on non-constant weights $\lambda$.
\end{remark}

\begin{remark}
At no point we can assume uniqueness of the flow, as the barycenter and the optimal transport plans are in general not unique. As an example, consider the  case of two atomic agents in the plane
\[
\mu^1=\frac{1}{2}(\delta_{(0,0)}+\delta_{(1,1)})\quad\mu^2=\frac{1}{2}(\delta_{(0,1)}+\delta_{(1,0)}).
\]
This configuration admits infinite barycentres, i.e., all convex combinations of the following two measures
\[
m^1=\frac{1}{2}(\delta_{(0,1/2)}+\delta_{(1,1/2)})\quad m^2=\frac{1}{2}(\delta_{(1/2,0)}+\delta_{(1/2,1)}),
\]
are barycenters with weights $\lambda = 1/2$.
It is clear that two solutions starting from the measure $\mu^1\otimes\mu^2$ are obtained as observed in Section \ref{SubSec:GeodesicSolutions} connecting both $\mu^1$ and $\mu^2$ to the same barycenter $m^i$, $i =1,2$.
Interestingly, there is no need to choose the same barycentre for different agents. This produces a third more interesting solution, illustrated in Figure \ref{fig:ex_uniqueness} (right).
\end{remark}

\begin{figure}
\includegraphics[width = 1\linewidth]{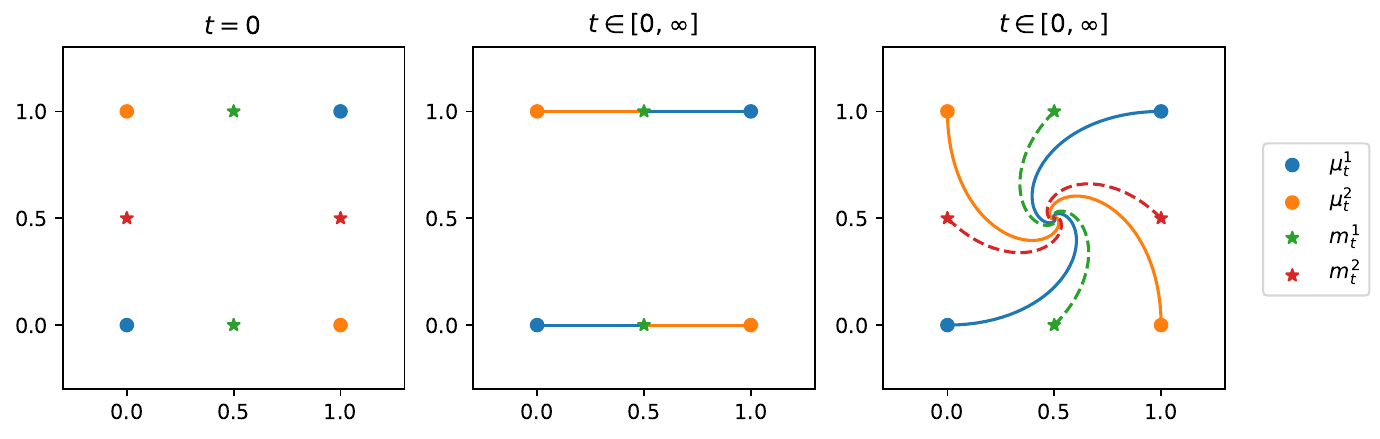}
\caption{Example of non-uniqueness of solutions to the MDI \eqref{eq:MDI}. On left, the initial agents positions $\mu^1_0,\mu_0^2$ and two (equally weighted) barycenters $m_0^1,m_0^2$. On the centre and on the right, trajectories of two solutions to the same MDI.}
\label{fig:ex_uniqueness}
\end{figure}

\section{Application to consensus dynamics for optimization} 
\label{sec:num}

In this section, we apply the proposed consensus multi-agent dynamics \eqref{eq:MDI} in the context of optimization, inspired by the class of Consensus-Based Optimization (CBO) \cite{pinnau2017consensus} methods developed in Euclidean settings. In the numerical experiments, we illustrate how the proposed multi-agent system is used as a computational tool to solve minimization problems for probability-valued objectives.

CBO algorithms exploit a system of interacting agents (also called \textit{particles}) to compute global minima of a given objective function. During the evolution, the agents instantaneously move towards a consensus point given by a weighted average of their positions. Different weights are assigned to the $N$ agents, depending on their objective value. In CBO, the weights are characterized by the Boltzmann--Gibbs distribution associated with $\E$.

Let $\E: \mathcal{P}_2(\Omega) \to \RR$ be a objective function, and consider the corresponding optimization problem
\begin{equation}
\mu^* \in \underset{\mu \in \mathcal{P}_2(\Omega)}{\textup{argmin}}\, \E(\mu)\,.
\label{eq:minpb}
\end{equation}
A CBO dynamics in $\mathcal{P}_2(\Rd)$, considers the consensus multi-agent system \eqref{eq:MDI} with  a particular weight function $\lambda = \lambda^\alpha$ given by the Boltzmann--Gibbs distribution 
\begin{equation}
\lambda^\alpha(\mu) := e^{-\alpha \E(\mu)}
\label{eq:cboweights}
\end{equation}
at temperature $1/\alpha$, for some parameter $\alpha>0$. 
We note that, for a collection of agents $\mu^1, \dots, \mu^N \in \mathcal{P}(\Rd)$, any sequence of barycenters
$
\overline\mu^\alpha \in \Bar(\lambda^\alpha \mathbb{P}_{\mu})$,  $\alpha>0
$
converges to the best agent among the ensemble 
\[
\overline\mu^\alpha \longrightarrow \underset{\mu^i,\, i=1, \dots,N}{\textup{argmin}} \,\E(\mu^i)\,, \qquad \textup{as} \quad \alpha \to \infty
\]
provided the (arg)minimum among the agents is uniquely determined. As the barycenter instantaneously moves (for $\alpha \gg 1)$ towards the agent with best location in terms of objective value, the multi-agent system is heuristically expected to create consensus in a neighborhood of a minimizer. We note that this argument is made rigorous for CBO in Euclidean settings in \cite{fornasier2021consensusbased}, by relying on a mean-field approximation of the dynamics.

To simulate the evolution of the multi-agent CBO system \eqref{eq:MDI}, we approximate each agent by $n$ particles (Dirac deltas) using an the initial data
\begin{equation*}
\mu^{i,n}_{0} = \frac 1n \sum_{\ell=1}^n \delta_{x_{0}^{i, \ell}}\,\qquad \textup{for}\quad i = 1, \dots,N\,, 
\end{equation*}
with $x_0^{i,\ell} \in \Rd$, $\ell = 1, \dots, n$, $i=1,\dots,N$, and use the time-discrete iteration \eqref{eq:iterativeMDI}.
We note that CBO methods in Euclidean space typically include a stochastic component in the agents evolution to favor exploration in the search space. In our experiments we also include stochasticity by perturbing the position of the particles $x^{\ell,i}_{(k)}$ at each step. The numerical strategy is illustrated in detail in Algorithm \ref{alg:cbo}.

\begin{algorithm}
\KwData{$\E$ (objective function), $N$ (\# agents), $n$ (\# particles per agent), $\tau$ (step size), $\sigma_1, \sigma_2>0$ (stochastic parameters), $\alpha>0$ (weights parameter), $k_{\max}$ (maximum iteration)}
\KwResult{$\overline{\mu}^\alpha_{(k)}$}

Sample $x_0^{i,\ell} \sim \mathcal{N}(0,I_\d)$ for all $i=1,\dots,N$ and $\ell=1,\dots,n$\;

$\mu^{i,n}_{(0)} = (1/n)\sum_{\ell = 1}^n \delta_{x_0^{i,\ell}}\;$ for all $i=1,\dots,N$

\For{$k = 0,1,\dots k_{\max}$}
{
Compute weights $\lambda^\alpha(\mu_{(k)}^{i,n}), i = 1, \dots, N$ according to \eqref{eq:cboweights}\;

Compute a $n$-particle approximation of a solution to the barycenter problem:\;

$\qquad \overline{\mu}_{(k)}^\alpha \approx \textup{argmin}_{\nu}
\sum_i \lambda^\alpha(\mu_{(k)}^{i,n})\, W_2^2(\nu, \mu_{(k)}^{i,n}),\quad \overline{\mu}^\alpha_{(k)} = (1/n)\sum_{\ell = 1}^n \delta_{\overline{x}_{(k)}^{\alpha}} 
$\;

\For{i = 1, \dots, N}{
Compute optimal coupling $G^i_{(k)}:[n] \to [n]$ between $\{x_{(k)}^{i,\ell}\}_{\ell=1}^n$ and $\{\overline{x}_{(k)}^{\alpha,\ell}\}_{\ell=1}^n$\;

Sample random vector $\xi_{(k)}^i \sim \mathcal{N}(0,I_\d)$\;

Compute stochastic strength $\sigma_{(k)}^{i} = \sigma_1 (W_2(\mu^{i,n}_{(k)}, \overline{\mu}^{\alpha,n}_{(k)}) + \sigma_2)$\;

Update agent:\;

$\quad x^{i,\ell}_{(k+1)} = x^{i,\ell}_{(k)} + \tau \left ( \overline{x}_{(k)}^{\alpha,G^i_{(k)}(\ell)}  - x^{i,\ell}_{(k)}\right) + \sigma_{(k)}^i \xi_{(k)}^i,\quad \ell = 1, \dots,n $\;

$\quad \mu^{i,n}_{(k+1)} = (1/n)\sum_{\ell = 1}^n \delta_{x_{(k+1)}^{i,\ell}}$\;
}
}
\caption{CBO algorithm for measure-valued, atomic agents}
\label{alg:cbo}
\end{algorithm}

\begin{figure}[h]
\begin{subfigure}{1\linewidth}
  \centering
  \includegraphics[width=1\linewidth]{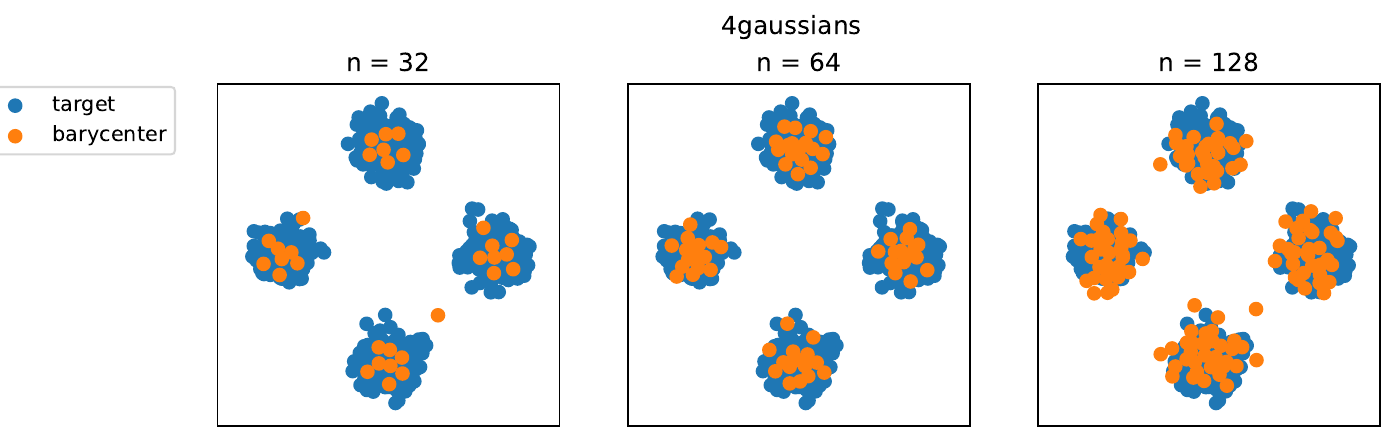}
\end{subfigure}
\\ 
\smallskip 
\begin{subfigure}{1\linewidth}
  \centering
  \includegraphics[width=1\linewidth]{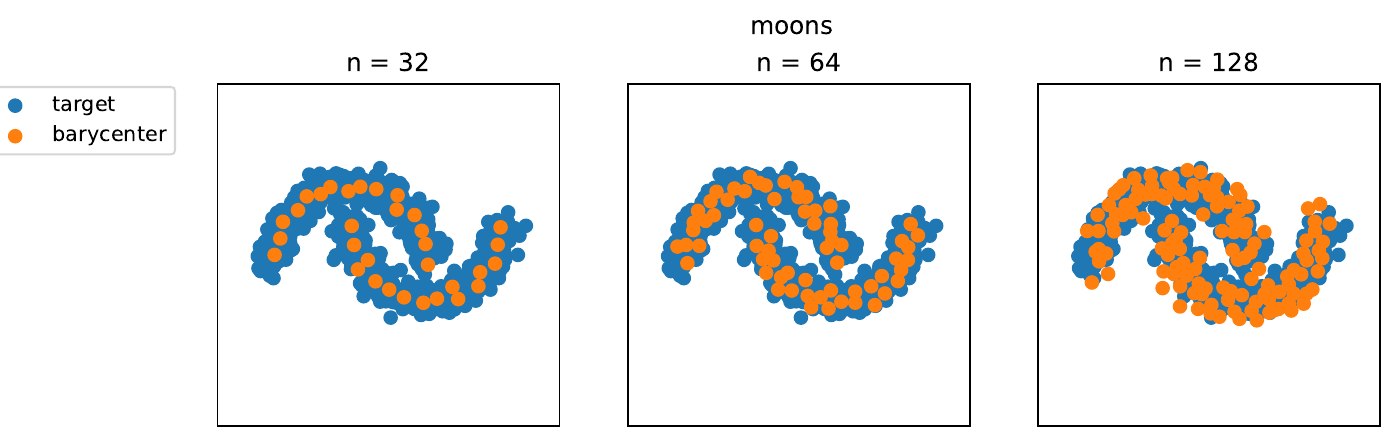}
\end{subfigure}%
 \\ 
\smallskip
\begin{subfigure}{1\linewidth}
  \centering
  \includegraphics[width=1\linewidth]{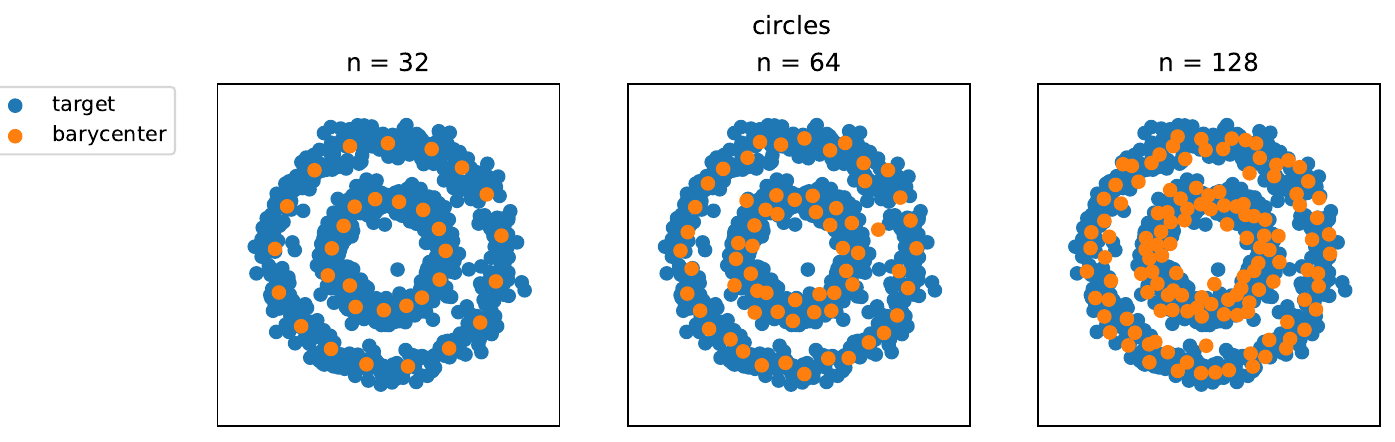}
\end{subfigure}%
\caption{Qualitatively results of Algorithm \ref{alg:cbo} for three different optimization problems. In blu, the target measure the agents aim to approximate, while in orange the computed solution (the final configuration of the barycenter). Parameters are set to $N = 30$, $n = 32,64,128$, $\lambda = 0.1$, $\sigma_1 = 0.3$, $\sigma_2 = 0.1$, $k_{\max} = 3000$, $\alpha = 10^6$. }
\label{fig:sol}
\end{figure}

To validate the algorithm, we consider the problem of approximating a target distributions $\nu^M\in \mathcal{P}(\mathbb{R}^2)$ made of $M=2000$ atomic points representing different 2-dimensional shapes. The objective, which we consider to be given as a black-box function, correspond to the Wasserstein distance from such target measure $\E(\cdot) = W_2(\cdot, \nu^M)$.
For the particle approximation of the agents, we employ different numbers $n$ of particles. The remaining algorithm parameters are set to $\tau=0.1, \sigma_1=0.3$, $\sigma_2 = 0.1$.  Figure \ref{fig:sol} shows the best agent in terms of objective $\E$, after $k_{\max} = 3000$ iterations for the three different shapes considered and for a varying numbers $n=32,64,128$ of particles per agent. We note that the algorithm is able to approximate the target shapes, independently on the number of particles used in the numerical implementation. For quantitatively results, we refer to Figure \ref{fig:obj} which shows the objective value of the best agent during the computation.  

\begin{remark} The update rule \eqref{eq:iterativeMDI} requires the computation of optimal tranport plans and barycenters between atomic measure for which many efficient algorithms have been developed in recent years. We refer to the monograph \cite{cuturi2019computational} for a detail review of the most established methods. In our experiment we used the \texttt{lp.free{\textunderscore}support{\textunderscore}barycenter}
function of the Python Optimal Transport Library \cite{flamary2021pot}, which is based on the algorithms proposed in \cite{cuturi14fast,alvarez2016fixed}. 

It is interesting to note, though, that approximation via particle systems or domain discretization is not the only way to numerically simulate the agents' dynamics. Several works, for instance, propose continuous optimal transport solvers with solutions parametrized by reproducing kernels \cite{genevay2016stochastic}, fully-connected neural networks  \cite{seguy2018large}, or   input convex neural networks \cite{Makkuva2020}.
\end{remark}

\begin{figure}[h]
  \centering
  \includegraphics[width=0.32\linewidth]{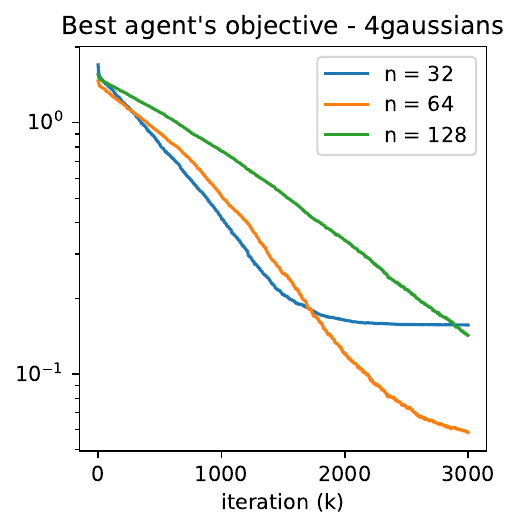}
  \includegraphics[width=0.32\linewidth]{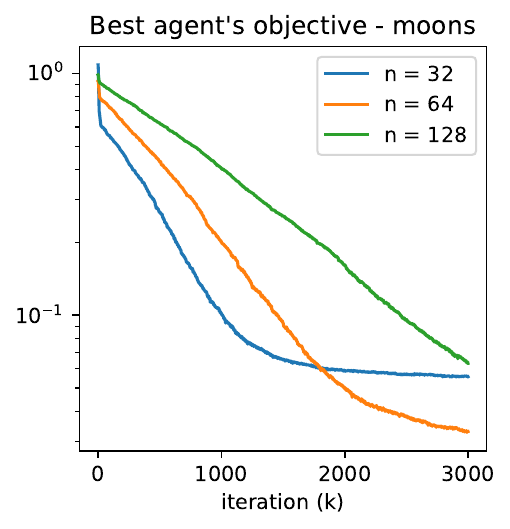}
  \includegraphics[width=0.32\linewidth]{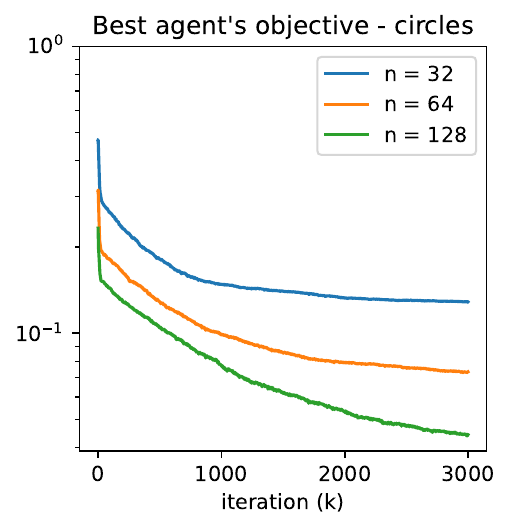}
\caption{Objective value of the best agent among the ensemble, as a function of iterative step, for the three different optimization problems considered.  Parameters are set to $N = 30$, $n = 32,62,128$, $\lambda = 0.1$, $\sigma_1 = 0.3$, $\sigma_2 = 0.1$, $k_{\max} = 3000$,  $\alpha = 10^6$. }
\label{fig:obj}
\end{figure}

\section{Conclusion and outlook}

In this work, we  presented a mathematical model for multi-agent dynamics of consensus type in the 2-Wasserstein space. The model describes agents moving along geodesics towards a common weighted barycenter.
Unlike previous works \cite{doucet2014,doucet2021,bullo2023}, we have considered a time-continuous evolution of the agents. The interaction is described by a system of Measure Differential Inclusions \cite{piccoli2018mdi}, and we have shown how to construct solutions both for absolutely continuous measures and for general ones, using an iterative scheme and under the assumption of compactly supported agents.

 Inspired by Consensus-Based Optimization \cite{pinnau2017consensus}, we applied the model to describe an optimization algorithm for black-box minimization problems over $\mathcal{P}_2(\Rd)$. The algorithm simply consists of the proposed multi-agent dynamics for a specific choice of the barycenter weights. We presented an application in the context of approximation of target measures showing the validity of the algorithm.

The work poses several interesting questions, both from theoretical and practical points of view which we plan to address. First of all, we aim to analyze the multi-agent system, in particular, to find sufficient conditions on the barycenter weights for the creation of consensus as $t \to \infty$ and to investigate the stability of eventual stationary states. To this aim, it may be beneficial to perform a kinetic or mean-field approximation of the agents' interaction. In Euclidean settings, this is a typical approach to studying multi-agent systems with large population size, through which one can gain a deeper mathematical understanding of the system dynamics, see, for instance, \cite{fetecau2011collective,pareschi13, pinnau2017consensus}.

After deriving a kinetic description of the multi-agent system, the proposed multi-agent system \eqref{eq:MDI} can be seen as a Monte Carlo approximation of the kinetic density evolution \cite{pareschi13}. This would allow us to employ well-known strategies in the simulation of particle systems to reduce the computational cost of the update step, such as random batch techniques \cite{AlPa,JLJ}. We remark that the reduction of the computational cost for the simulation of the multi-agent systems is of paramount importance for applications, as the dynamics require solving $\mathcal{O}(N)$ optimal transport problems per iterative step.
Lastly, we aim to include stochastic behavior in the proposed consensus multi-agent system. This is particularly relevant, for instance, in the context of optimizing dynamics,  where stochasticity is needed for the exploration of the search space.

\bibliographystyle{siamplain}
\bibliography{bibfile}


\appendix

\section{Proof of Proposition \ref{p:uppersemi}}
\label{appendix}

The aim of this section is to prove upper semi-continuity of the MPVF 
\[\mathcal{V} =  \bigotimes_{i=1}^N (\pi_0, \pi_1 - \pi_0)_\#  \Gamma_o\left(\mu^i, \Bar\left( \lambda \mathbb{P}_\mu\right)\right)\]
introduced in \eqref{def:V}. We will show that the map $\Gamma_o$ is upper semi-continuous in all variables in case of a multi-marginal optimal transport problem with continuous cost. Since the barycenter is the solution to a multi-marginal transport problem as shown in \cite{santambrogio2015optimal}, and the push-forward is continuous, their compositions is upper semi-continuous.
Given that these asserts are mostly folklore and the closest we have is in the context of weak topology and the Wasserstein distance in \cite{ambrosio2005gradient}, we give here a precise exposition here with a more \textit{ad hoc} statements.

\begin{lemma}
Given $X$ and $Y$ two Polish spaces, then for any $\mu\in\mathcal{P}(X)$ and any $\nu\in\mathcal{P}(Y)$ the set $\Gamma(\mu,\nu)$ is compact.
\end{lemma}
\begin{proof}
Given $\mu\in\mathcal{P}(X)$, $\nu\in\mathcal{P}(Y)$, and a sequence $\gamma^n\in\Gamma(\mu,\nu)$ we get that for any $\varepsilon>0$ we can find two compact sets $K_1\subseteq X$ and $K_2\subseteq Y$ such that
\[
\mu(K_1^c)\leq \varepsilon \quad \nu(K_2^c)\leq \varepsilon
\]
from which we get
\begin{align*}
\gamma^n((K_1\times K_2)^c)&\leq \gamma^n(K_1^c\times Y)+\gamma^n(X\times K_2^c)\\
&=\mu(K_1^c)+\nu(K_2^c)\\
&\leq 2\varepsilon. 
\end{align*}
The above proves that $\gamma^n$ is a tight family of probability measures and thanks to Prokhorov's Theorem \cite{ambrosio2005gradient}, we can find a converging subsequence and thus proving our thesis.
\end{proof}

\begin{lemma}\label{Lem:minimosuinsiemi}
Given a metric space $M$ and a continuous function $f:M\to\RR$.
We call $\mathcal{K}(M)$ the metric space of all compact sets with the Hausdorff metric.
Then the function
\begin{align*}
F:\mathcal{K}(M)&\to \RR\\
K&\mapsto \min_{x\in K}f(x)
\end{align*}
is lower semi-continuous.
\end{lemma}
\begin{proof}
Given a sequence $K_n\in\mathcal{K}(M)$ converging to $K\in\mathcal{K}(M)$ in the Hausdorff metric means that any Cauchy sequence $x_n\in K_n$ converges to an $x\in K$ and for any $x\in K$ there is a sequence $x_n\in K_n$ converging to it.
Given any converging sequence $x_n\in K_n$ to an element $ x\in K$ we have
\[
F(K)\leq f(x) = \lim_{n\to\infty}f(x_n)
\]
from which we can deduce
\[
\liminf_{n\to\infty} F(K_n)\geq F(K).
\]
Now let us take $x\in K$ such that $f(x)=F(K)$ and consider $x_n\in K_n$ such that $x_n\to x$ and we can compute
\[
F(K)=\lim_{n\to\infty} f(x_n)\geq \liminf_{n\to\infty} F(K_n)\geq F(K)
\]
thanks to which we have
\[
\liminf_{n\to\infty} F(K_n)= F(K).
\]

\end{proof}

\begin{lemma}
Given $X$ and $Y$ separable metric spaces then the map
\begin{align*}
\Gamma : \mathcal{P}_2(X)\times\mathcal{P}_2(Y)&\to \mathcal{K}(\mathcal{P}(X\times Y))\\
(\mu,\nu)&\mapsto\Gamma(\mu,\nu)
\end{align*}
is continuous.
\end{lemma}

\begin{proof}
Given converging sequences $\mu_n\in\mathcal{P}(X)$ and $\nu_n\in\mathcal{P}(Y)$ to $\mu\in\mathcal{P}(X)$ and $\nu\in\mathcal{P}(Y)$ respectively we show the necessary and sufficient properties to get Hausdorff convergence.
For the first part we are given $\gamma^n\in\Gamma(\mu_n,\nu_n)$ such that $\gamma^n\to\gamma$ in the weak topology we get that for $\varphi\in C(X)$ we compute
\[
\int_X\varphi\mu=\lim_n\int_X\varphi\mu^n=\lim_n\int_{X\times Y}\varphi\circ\pi_1\gamma^n= \int_{X\times Y}\varphi\circ\pi_1\gamma
\]
and because it is the same also for the second component we have $\gamma\in\Gamma(\mu,\nu)$.
Now for the second step we consider $\gamma\in\Gamma(\mu,\nu)$ and we define $\tau^n\in\Gamma_o(\mu,\mu_n)$ and $\sigma^n\in\Gamma_o (\nu,\nu_n)$ such that
\[
W_2(\mu,\mu_n)^2=\int d_X(x_1,x_2)^2\tau^n(dx_1,dx_2)\quad W_2(\nu,\nu_n)^2=\int d_Y(y_1,y_2)^2\sigma^n(dy_1,dy_2)\,.
\]
Now we disintegrate both $\tau^n$ and $\sigma^n$ over $\mu$ and $\nu$ respectively to obtain Borel collections of probability measures $\tau^n_x$ and $\sigma^n_y$ such that for any $\eta\in C(X\times X)$ and $\xi\in C(Y\times Y)$
\begin{align*}
\int_{X\times X}\eta(x_1,x_2)\tau^n(dx_1,dx_2)&=\int_X \int_{X }\eta(x_1,x_2)\tau^n_{x_1}(dx_2) \mu(dx_1)\\
\int_{Y\times Y}\xi(y_1,y_2)\sigma^n(dy_1,dy_2)&=\int_Y \int_{Y}\xi(y_1,y_2)\sigma^n_{y_1}(dy_2) \nu(dy_1)
\end{align*}
and use them to define a sequence of measures $\gamma^n\in\Gamma(\mu_n,\nu_n)$ as
\[
\int_{X\times Y}\eta\gamma^n=\int_{X\times Y}\left(\int_{X\times Y}\eta(w,z)(\tau^n_{ x}\otimes \sigma^n_{ y})(dw,dz)\right)\gamma(d x,d y)\quad \eta\in\Lip(X\times Y)
\]
and show that it converges to $\int\eta\gamma$.
We compute
\begin{align*}
\Big|\int_{X\times Y}\eta\gamma^n&-\int_{X\times Y}\eta\gamma\Big| \\
&=\left|\int_{X\times Y}\left(\int_{X\times Y}\eta(w,z)-\eta(x,y)(\tau^n_{ x}\otimes \sigma^n_{ y})(dw,dz)\right)\gamma(d x,d y)\right|\\
&\leq \Lip(\eta)\int_{X\times Y}\left(\int_{X\times Y}d_X(w,x)+d_Y(z,y)(\tau^n_{ x}\otimes \sigma^n_{ y})(dw,dz)\right)\gamma(d x,d y)\\
&\leq \Lip(\eta)\left(W_2(\mu,\mu_n)+ W_2(\nu,\nu_n)\right)
\end{align*}
thanks to which we conclude the Lemma.
\end{proof}

\begin{corollary}\label{Cor:USCTransport}
Given $X$ and $Y$ separable metric spaces, and a continuous function $c:X\times Y\to\RR$, then the multivalued function
\begin{align*}
\Gamma_o : \mathcal{P}(X)\times\mathcal{P}(Y)&\to \mathcal{K}(\mathcal{P}(X\times Y))\\
(\mu,\nu)&\mapsto\underset{\gamma\in\Gamma(\mu,\nu)}{\mathrm{argmin}} \int_{X\times Y} c\gamma
\end{align*}
is upper semi-continuous.
\end{corollary}
\begin{proof}
Given converging sequences $\mu_n\in\mathcal{P}(X)$ and $\nu_n\in\mathcal{P}(Y)$ to $\mu\in\mathcal{P}(X)$ and $\nu\in\mathcal{P}(Y)$ respectively, and a sequence $\gamma^n\in\Gamma_o(\mu_n,\nu_n)$ such that $\gamma^n\to\gamma\in\Gamma(\mu,\nu)$.
We define the function
\begin{align*}
W : \mathcal{P}(X)\times\mathcal{P}(Y)&\to \RR \\
(\mu,\nu)&\mapsto\min\left\{ \int_{X\times Y} c\gamma\,\middle|\, \gamma\in \Gamma(\mu,\nu)\right\}
\end{align*}
and from all the preceding Lemmas combined we have that
\[
\liminf_{n\to\infty}W(\mu_n,\nu_n)=W(\mu,\nu)
\]
and since we also have $W(\mu_n,\nu_n)=\int c\gamma^n\to\int c\gamma$ because $c$ is continuous we can conclude that
\[
W(\mu,\nu)=\int c\gamma
\]
and therefore $\gamma\in\Gamma_o(\mu,\nu)$.
\end{proof}

\begin{corollary}
Given a continuous function $c:(\Rd)^N\to\RR$ then the multivalued map
\[
\Gamma_o:\left(\mathcal{P}_2(\Rd)\right)^N\rightrightarrows\mathcal{P}_2((\Rd)^N)
\]
defined as
\[
\Gamma_o(\mu_1,\dots,\mu_N):=\underset{\gamma\in\Gamma(\mu_1,\dots,\mu_N)}{\mathrm{argmin}} \int_{(\Rd)^N}c\gamma\,
\]
is upper semi-continuous.
\end{corollary}

\begin{proof}
Apply Corollary \ref{Cor:USCTransport} to the spaces $\Rd$ and $(\Rd)^{N-1}$.
\end{proof}

\begin{proof}[Proof of Proposition \ref{p:uppersemi}]
The proof follows from the preceding corollaries because the push-forward is continuous, the map $\Gamma_o$ which is optimal for the quadratic distance is upper semi-continuous, and $\Bar(\cdot)$ is upper semi-continuous because it is the solution to a multi-marginal optimal transport problem as shown in \cite{santambrogio2015optimal}.
\end{proof}

\end{document}